\documentclass{mfatshort}

\newtheorem{theorem}{Theorem}[section]
\newtheorem{corollary}[theorem]{Corollary}

\theoremstyle{definition}
\newtheorem{definition}[theorem]{Definition}
\theoremstyle{remark}
\newtheorem{remark}[theorem]{Remark}
\numberwithin{equation}{section}
\renewenvironment {proof} {\begin{trivlist} \item[\hspace{\labelsep}%
\sc Proof.]}{$\Box$ \end{trivlist}}

   \def\cN{{\mathcal N}}

\renewcommand {\phi}{\varphi}          

\newcommand {\supp}{\mathop {\rm supp}}     

\begin{document}

\title[Multidimensional moment problem  and Stieltjes transform]
{Multidimensional moment problem  and Stieltjes transform}

\author[I. Kovalyov]{Ivan Kovalyov}
\address{Universität Osnabrück\\
Albrechtstr. 28a  \\
49076 Osnabrück \\
Germany.}
\email{i.m.kovalyov@gmail.com}

\thanks{I. Kovalyov also gratefully acknowledges financial support by the German Research Foundation (DFG, grant 520952250).}

\subjclass[2010]{Primary 30E05 ; Secondary 44A60; 30B70 ; 14P99}
\keywords{Continued fractions, multidimensional moment problem, Schur algorithm.}

\begin{abstract} The truncated multidimensional moment problem is studied in terms of the Stieltjes transform as the interpolation problem. A step-by-step algorithm is constructed for the multidimensional moment problem and the set of solutions is found in terms of  continued fractions.

\end{abstract}

\maketitle
\tableofcontents
\section{Introduction}
 Let $\mu$ be a nonnegative Borel measure on $\mathbb{R}^n$, where $\supp(\mu)=A\subseteq \mathbb{R}^n$. The moment sequence 
 $\mathbf{s}=\{s_{i_1,\ldots,i_n}\}_{i_1,\ldots, i_n=0}^\ell$  is defined by
\begin{equation}\label{14.2.1}
s_{i_1,...,i_n}=\int_{\mathbb{R}^n} t_1^{i_1}\times\ldots\times  t_n^{i_n}d\mu(t_1,...,t_n).
\end{equation}

On the other hand, we can reformulated \eqref{14.2.1} in terms of the Stieltjes transform as an interpolation problem.
Recall  a $n-$variate Stieltjes transform (see \cite{Cuyt2004})
\begin{equation}\label{20.int. st tr.1}
\int_{\mathbb{R}^n}\cfrac{d\mu(t_1,t_2,..., t_n)}{1+ \sum\limits_{i=1}^n t_i\tilde{z}_i}.
\end{equation}

Setting $\tilde z_i=-\cfrac{1}{z_i}$, we define an  associated function $F$ by
\begin{equation}\label{20.2.2}
 F(z_1,z_2,...,z_n)=-\frac{1}{\prod\limits_{i=1}^n z_i}\int_{\mathbb{R}^n}\cfrac{d\mu(t_1,t_2,..., t_n)}{1- \sum\limits_{i=1}^n \cfrac{t_i}{z_i}}.
\end{equation}   

\begin{definition}(\cite{Luger19}) A Stolz domain with centre $t_0\in \mathbb{R}$ and angle $\theta\in(0,\frac{\pi}{2}]$ in the set

\[
\{z\in \mathbb{C}^+ |\theta<\arg( z-t_0)<\pi-\theta\}
\]
The symbol   $z\widehat{\rightarrow}t_0$ denotes the limit $z{\rightarrow}t_0$ in any Stolz domain with centre $t_0$ and the symbol $z\widehat{\rightarrow}\infty$  denotes the limit
$|z|{\rightarrow}\infty$ in any Stolz domain with centre $0$.
\end{definition}

Moreover, \eqref{20.2.2} admits the following asymptotic expansion
\begin{equation}\label{20.2.3}\begin{split}
    F(z_1,...,z_n)&=-\frac{1}{\prod\limits_{i=1}^n z_i}\int_{\mathbb{R}^n}\cfrac{d\mu(t_1,t_2,..., t_n)}{1- \sum\limits_{i=1}^n \cfrac{t_i}{z_i}}=
    \\&=-\frac{1}{\prod\limits_{i=1}^n z_i} \int_{\mathbb{R}^n}\sum_{k=0}^\infty\left( \sum\limits_{i=1}^n \cfrac{t_i}{z_i}\right)^k d\mu(t_1,t_2,..., t_n)=\\&=
    -\int_{\mathbb{R}^n} \sum_{k=0}^\infty\!\! \!\!\!\sum_{\begin{matrix} 
  \alpha_i^k\geq0  \\
  \alpha_1^k+\ldots \alpha_n^k=k \\
   \end{matrix} }\!\!\!\!\!\!\!\!\!\!\binom{k}{\alpha_1^k, \ldots, \alpha_n^k}\cfrac{t_1^{\alpha_1^k}t_2^{\alpha_2^k}\ldots t_n^{\alpha_n^k}d\mu(t_1,t_2,..., t_n)}{z_1^{\alpha_1^k+1}z_2^{\alpha_2^k+1}\ldots z_n^{\alpha_n^k+1}}=\\&=
    -\sum_{k=0}^\infty \!\!\!\sum_{\begin{matrix} 
  \alpha_i^k\geq0  \\
  \alpha_1^k+\ldots \alpha_n^k=k \\
   \end{matrix} }\!\!\!\!\!\binom{k}{\alpha_1^k, \ldots, \alpha_n^k}\cfrac{s_{\alpha_1^k,\alpha_2^k,\ldots, \alpha_n^k}}{z_1^{\alpha_1^k+1}z_2^{\alpha_2^k+1}\ldots z_n^{\alpha_n^k+1}}, 
\quad z_i\widehat{\rightarrow}\infty.
\end{split}
\end{equation}

However, we can reformulate the moment problem \eqref{14.2.1}  in terms of the generalized moments. Define a linear functional $\mathfrak{S}$  on the monomials $t_1^{\alpha_1}\times\ldots\times  t_n^{\alpha_n}$ by
\begin{equation}\label{14.2.3S}
	\mathfrak{S}(t_1^{\alpha_1}\times\ldots\times  t_n^{\alpha_n})=s_{\alpha_1,...,\alpha_n}.
\end{equation}

In the general case, the associated function can be defined by
\begin{equation}\label{21p.1.11s}
\begin{split}
    F(z_1,z_2,...,z_n)&=-\frac{1}{\prod\limits_{i=1}^n z_i} \mathfrak{S}\left(   \cfrac{1}{1- \sum\limits_{i=1}^n \cfrac{t_i}{z_i}}\right)=
   -\frac{1}{\prod\limits_{i=1}^n z_i}\mathfrak{S}\left(\sum_{k=0}^\infty\left( \sum\limits_{i=1}^n \cfrac{t_i}{z_i}\right)^k \right)=\\&=\!
    -\mathfrak{S}\!\!\left(  \sum_{k=0}^\infty \!\!\!\sum_{\begin{matrix} 
  \alpha_i^k\geq0  \\
  \alpha_1^k\!+\!\ldots\!+\! \alpha_n^k\!=\!k \\
   \end{matrix} }\!\!\! \binom{k}{\alpha_1^k, \ldots, \alpha_n^k}\!\cfrac{t_1^{\alpha_1^k}t_2^{\alpha_2^k}\ldots t_n^{\alpha_n^k}}{z_1^{\alpha_1^k+1}z_2^{\alpha_2^k+1}\ldots z_n^{\alpha_n^k+1}}
   \right)=\\&=
    -\sum_{k=0}^{n\times \ell}\!\!\! \sum_{\begin{matrix} 
  \alpha_i^k\geq0  \\
  \alpha_1^k\!+\ldots \!+\!\alpha_n^k=k \\
   \end{matrix}}\!\!\!\binom{k}{\alpha_1^k, \ldots, \alpha_n^k}\cfrac{s_{\alpha_1^k,\alpha_2^k,\ldots, \alpha_n^k}}{z_1^{\alpha_1^k+1}z_2^{\alpha_2^k+1}\ldots z_n^{\alpha_n^k+1}}+\\&  +o\!\left( \!\!\sum_{\begin{matrix} 
  \alpha_i^{n\times \ell}\geq0  \\
  \alpha_1^{n\!\times\! \ell}\!+\!\ldots\!+\! \alpha_n^{n\!\times\! \ell}\!=\!{n\!\times\! \ell} \\
   \end{matrix}}\cfrac{1}{z_1^{\alpha_1^{n\!\times\! \ell}+1}\ldots z_n^{\alpha_n^{n\!\times\! \ell}+1}}\right), \quad z_i\widehat{\rightarrow}\infty.\end{split}
\end{equation}

Furthermore, the Stieltjes transform can be applicable  to the truncated moment problem for atomic measure. Let  $\mu$ be an atomic measure on $\mathbb{R}^n$, where sup$(\mu)=A\subseteq \mathbb{R}^n$ and 
 \begin{equation}\label{24p*.2.1}
 	\mu(x)=\sum_{k=0}^Mm_k\delta_{x_k}(x), \quad m_k>0.
 \end{equation}
 In this case, the moment sequence 
 $\mathbf{s}=\{s_{i_1,\ldots,i_n}\}_{i_1,\ldots, i_n=0}^\ell$  is generated by
\begin{equation}\label{24q.2.1}
s_{i_1,...,i_n}=\int_{\mathbb{R}^n} x_1^{i_1}\times\ldots\times  x_n^{i_n}d\mu(x_1,...,x_n)=\sum_{k=0}^M m_k t_1^{i_1}\times\ldots\times  t_n^{i_n}
\end{equation}
and  the associated function takes the following asymptotic expansion 

\begin{equation}\label{20.2.3e}\begin{split}
    F(z_1,...,z_n)&=-\frac{1}{\prod\limits_{i=1}^n z_i}\int_{\mathbb{R}^n}\cfrac{d\mu(x_1,x_2,..., x_n)}{1- \sum\limits_{i=1}^n \cfrac{x_i}{z_i}}=
    \\& =-\sum_{k=0}^{n\times \ell}\!\!\! \!\!\!\sum_{\begin{matrix} 
  \alpha_i^k\geq0  \\
  \alpha_1^k\!+\ldots \!+\!\alpha_n^k=k \\
   \end{matrix}}\!\!\!\!\!\!\binom{k}{\alpha_1^k, \ldots, \alpha_n^k}\cfrac{s_{\alpha_1^k,\alpha_2^k,\ldots, \alpha_n^k}}{\prod\limits_{j=1}^nz_j^{\alpha_j^k+1}}+\\&  +o\!\left( \!\!\sum_{\begin{matrix} 
  \alpha_i^{n\times \ell}\geq0  \\
  \alpha_1^{n\!\times\! \ell}\!+\!\ldots\!+\! \alpha_n^{n\!\times\! \ell}\!=\!{n\!\times\! \ell} \\
   \end{matrix}}\cfrac{1}{\prod\limits_{k+1}^n z_k^{\alpha_k^{n\!\times\! \ell}+1}}\right), \quad z_i\widehat{\rightarrow}\infty.
   \end{split}
   \end{equation}

 \textbf{ The multidimensional  problem} $\bold{MP}(\mathbf{s},\ell)$ can be formulated as:
 
 Given a sequence of real numbers $\mathbf{s}=\{s_{i_1,i_2,...,i_n}\}_{i_1,...,i_n=0}^\ell$. Describe the set of function $F$ such that
\begin{equation}\label{20.2.3wq}\begin{split}
 	 F(z_1,z_2,...,z_n)&=  -\sum_{k=0}^{n\times\ell} \sum_{\begin{matrix} 
  \alpha_i^k\geq0  \\
  \alpha_1^k+\ldots \alpha_n^k=k \\
   \end{matrix} }\binom{k}{\alpha_1^k, \ldots, \alpha_n^k}\cfrac{s_{\alpha_1^k,\alpha_2^k,\ldots, \alpha_n^k}}{z_1^{\alpha_1^k+1}z_2^{\alpha_2^k+1}\ldots z_n^{\alpha_n^k+1}}+\\&
   +o\!\left( \!\!\sum_{\begin{matrix} 
  \alpha_i^{n\times \ell}\geq0  \\
  \alpha_1^{n\!\times\! \ell}\!+\!\ldots\!+\! \alpha_n^{n\!\times\! \ell}\!=\!{n\!\times\! \ell} \\
   \end{matrix}}\cfrac{1}{z_1^{\alpha_1^{n\!\times\! \ell}+1}\ldots z_n^{\alpha_n^{n\!\times\! \ell}+1}}\right), \quad z_i\widehat{\rightarrow}\infty.
\end{split}
 \end{equation}
 
 If $\ell$ is finite, then $\bold{MP}(\mathbf{s},\ell)$ is called a truncated problem, otherwise one is called a full problem.

 However, the sequence  $\mathbf{s}=\{s_{i_1,i_2,...,i_n}\}_{i_1,...,i_n=0}^\ell$ admits the following decomposition
 \begin{equation}\label{21.1.13s}
	 \mathbf{s}=\bigcup_{j_2,\ldots, j_n=0}^\ell  \mathbf{s}_{j_2,\ldots, j_n}, \quad \mbox{where}\quad \mathbf{s}_{j_2,\ldots, j_n}=\{s_{i,j_2,\ldots, j_n}\}_{i=0}^\ell, \quad j_2,\ldots, j_n=\overline{0,\ell}.
  \end{equation}
  
In addition, if we set
 \begin{equation}\label{21.1.13sw3}
	\mathfrak s_{i,j_2,\ldots, j_n}=\binom{i+j_2+\ldots+ j_n}{i,j_2, \ldots,j_n}s_{i,j_2,\ldots, j_n},
  \end{equation}
  then \eqref{20.2.3wq} can be rewritten as
  \begin{equation}\label{20.2.3wq1}\begin{split}
   	 F(z_1,z_2,...,z_n)&=  -\sum_{j_2,\ldots, j_n=0}^\ell \cfrac{1}{z_2^{j_2+1}\cdot\ldots\cdot z_n^{j_n+1}}F_{j_2,\ldots, j_n}(z_1),  
   \end{split}
 \end{equation}
  where   $F_{j_2,\ldots, j_n}$  admits the following asymptotic expansion
    \begin{equation}\label{20.2.3wq2}
    	F_{j_2,\ldots, j_n}(z_1)=-\cfrac{ \mathfrak{s}_{0,j_2,\ldots, j_n}}{z_1}-\cfrac{ \mathfrak{s}_{1,j_2,\ldots, j_n}}{z_1^2}-\ldots-\cfrac{ \mathfrak{s}_{\ell,j_2,\ldots, j_n}}{z_1^{\ell+1}}+o\left(\cfrac{1}{z_1^{\ell+1}}\right).
     \end{equation}
     
     $F_{j_2,\ldots, j_n}$   is called an  associated function of the  sequence $\mathfrak s_{j_2,\ldots, j_n}=\{ \mathfrak{s}_{i,j_2,\ldots, j_n}\}_{i=0}^\ell$.
 
 The representation \eqref{20.2.3wq1}  is more convenient in terms of    $F_{j_2,\ldots, j_n}$  functions.  In this case,  the multidimensional problem is rewritten as the union of  one-dimensional problems.  Hence the problem becomes simpler. 
  
 In Section 2, we recall the basic results of the one-dimensional moment problem, which will be applied to the multidimensional problem. In Section 2 and Section 3, we study truncated multidimensional moment problem in terms of the Stieltjes transform, the step-by-step algorithms are found. In this case, the solutions of the truncated moment problems can be represented in terms of \textbf{P}-fractions and  \textbf{S}-fractions.  Finally, in Section 4, we consider the full problem and describe the set of solutions. The description of solutions to the full problem is based on the results of the previous sections. 

 \section{Indefinite one-dimensional moment problem}
In this section, we recall some results of the one-dimensional moment problem in terms of the Stieltjes transform, which provide a basis for future results.

Let $\mathbf{s}=\{s_{j}\}_{j=0}^{\ell}$ be a sequence
of real numbers and let a functional $\mathfrak{S}$ be defined on
the linear space
$\mathcal{P}=span\left\{z^{j}: j=\overline{0,\ell}\right\}$
by 
\begin{equation}\label{P3.mom}s_j=\mathfrak{S}(z^{j}).\end{equation}

We consider the general case, i.e. the measure may not exist. Hence, the problem is studied as the interpolation problem. Let us define it:

 \textbf{Moment problem $MP(\mathbf{s},\ell)$}:
Given a sequence of real numbers $\mathbf{s}=\{s_{j}\}_{j=0}^{\ell}$. Find the set of function $F$, which admit the  asymptotic expansion 
\begin{equation}\label{14.int.th_1}
	 f(z)=-\frac{s_{0}}{z}-\frac{s_{1}}{z^2}-\cdots-\frac{s_{\ell}}{z^{\ell+1}}
    +o\left(\frac{1}{z^{\ell+1}}\right),\quad\quad z\widehat{\rightarrow}\infty.
\end{equation}

In addition, we  recall  a set $\mathcal{N}(\textbf{s})=\{n_{j}\}_{j=1}^{N}$ of
normal indices of  $\textbf{s}=\{s_j\}_{j=0}^{\ell}$, which  is defined by
\begin{equation}\label{3p.2.4}
\mathcal{N}(\textbf{s})=\{n_{j}: D_{n_{j}}\neq0, j=1,2,\ldots,
N\},\quad D_{n_{j}}:=\textup{det}({s}_{i+k})_{i,k=0}^{n_{j}-1}.
\end{equation}

So, let us formulate the main results of $MP(\mathbf{s},\ell)$. The following statement includes results in terms of \textbf{P}-fractions.

\begin{theorem}\label{14p.int.th_ind} (\cite{Der03}) Let $\mathbf{s}=\{s_{j}\}_{j=0}^{2n_N-1}$ be a sequence of real numbers and  let $\cN({\mathbf s})=\{n_j\}_{j=1}^N$ be the set of normal indices of  ${\mathbf
s}$. Then any solution of the moment problem $MP(\mathbf{s}, 2n_N-1)$ takes the form
\begin{equation}\label{14.int.th_2}
	f(z)=-\cfrac{b_0}{a_0(z)-\cfrac{b_1}{a_1(z)-\ldots-\cfrac{b_{N-1}}{a_{N-1}(z)+\tau(z)}}},
\end{equation}
where the prameter $\tau$ satisfies the following condition
\begin{equation}\label{14.int.th_3wx}
	\tau(z)=o(1),\quad z\widehat{\rightarrow}\infty,
\end{equation}
and atoms $(a_i,b_i)$ can be calculated by
\begin{equation}\label{14.int.th_3}
	\begin{split}
	&b_0=s _{n_1-1},\quad b_j=s^{(j)}_{\nu-1} \quad  \mbox{and}\quad
	 a_{j}(z)=\cfrac{1}{D^{(j-1)}_{\nu}} \begin{vmatrix}
	\mathfrak  s^{(j-1)}_{0}&\mathfrak  s^{( j-1)}_{1}& \ldots&  \mathfrak s^{(j-1)}_{\nu}\\
	  \ldots& \ldots&\ldots& \ldots\\
	  s^{(j-1)}_{\nu-1}& \mathfrak s^{( j-1)}_{\nu}& \ldots&\mathfrak  s^{( j-1)}_{2\nu-1}\\
	    1& z&\ldots& z^{\nu}
	 \end{vmatrix}
	 ,\\&
	s^{( j)}_i=\cfrac{(-1)^{i+\nu}}{\left(s^{(j-1)}_{\nu-1}\right)^{i+\nu+2}}\begin{vmatrix}
 s^{(j-1)}_{\nu}& s^{( j-1)}_{\nu-1} &0&\ldots& 0\\
\vdots& \ddots &\ddots&\ddots & \vdots\\
\vdots&  &\ddots&\ddots &0\\
\vdots&  & &\ddots &s^{(j_2,...,j_n, j-1)}_{\nu-1}\\
s^{(j-1)}_{2\nu+i}& \ldots &\ldots&\ldots& s^{( j-1)}_{\nu-1}
\end{vmatrix},
\end{split}\end{equation}
where $i=\overline{0,2\nu-1}$, $j=\overline {0, N-1}$, $s^{(0)}_{i}=s_{i}$,  $\nu=n_j-n_{j-1}$ and $n_{0}=0$.

Furthermore, \eqref{14.int.th_2} can be represented by
\begin{equation}\label{14.int.th_3,444r}
 f(z)=-\cfrac{Q_{n_{N-1}}(z)\tau(z)+Q_{n_N}(z)}{P_{n_{N-1}}(z)\tau(z)+P_{n_N}(z)},
\end{equation}
where $P_{n_j}$ and $Q_{n_j}$ is called the polynomials of the first and second kind, respectively, which can be defined as the solutions of the following system
\begin{equation}\label{system1}
{b}_{j}y_{{n}_{j-1}}(z)-a_{j}(z)y_{{n}_{j}}(z)+y_{{n}_{j+1}}(z)=0\end{equation}
subject to the initial conditions
\begin{equation}\label{system1.1} P_{{n}_{-1}}(z)\equiv0,\mbox{ }P_{{n}_{0}}(z)\equiv1,\mbox{ }
Q_{{n}_{-1}}(z)\equiv-1,
\mbox{ }Q_{{n}_{0}}(z)\equiv0.\end{equation}
\end{theorem}
\begin{remark} We can rewrite \eqref{14.int.th_2} in  short form as follows
\[\underset{i=0}{\overset{N-1}{\mathbf{K}}}\left(-\cfrac{b_{i}}{a_{i}(z)}\right):=-\cfrac{b_0}{a_0(z)-\ldots-\cfrac{b_{N-1}}{a_{N-1}(z)+\tau(z)}}.
\]
\end{remark}

By \cite{DK15, DK17, K17}, the solutions of $MP(\mathbf{s},\ell)$ can be represented in terms of the  \textbf{S}--fractions. We consider the special case, where the sequence ${\mathbf s}=\{s_i\}_{i=0}^\ell$ is regular. Let us define it:

 \begin{definition} \label{def:3p.4.3}(\cite{DK15})
Let ${\mathbf s}=\{s_i\}_{i=0}^\ell$  be a sequence of real numbers and let $\mathcal{N}({\mathbf s})=\{n_j\}_{j=1}^N$ be the set of normal indices of ${\mathbf s}$.
A sequence ${\mathbf s}$ is called a regular,
if one of the following equivalent conditions  holds:
\begin{enumerate}
  \item [(1)] $P_{{n}_{j}}(0)\neq0$ for every $j\le N$;
  \item [(2)] $D_{{n}_{j}-1}^{+}\neq0$ for every $j\le N$;
  \item [(3)] $D_{{n}_{j}}^{+}\neq0$ for every $j\le N$, where $D^{+}_n=\mbox{det}(s_{i+j+1})_{i,j=0}^{n-1}$.
 \end{enumerate}
\end{definition}

 \begin{theorem}\label{2p.pr.alg1}(\cite{DK15,DK17})
Let ${\mathbf
s}=\{s_i\}_{i=0}^{2n_N-2}$ be a regular sequence   and  let
$\cN({\mathbf s})=\{n_j\}_{j=1}^N$ be the set of normal indices of ${\mathbf
s}$. Then any solution of  $MP(\mathbf{s}, 2n_N-2)$
 admits the  representation 
 \begin{equation}\label{5p.eq:3.6*9}
    f(z)= \frac{1}{\displaystyle -z m_{1}(z)+\frac{1}{\displaystyle
    l_{1}+\frac{1}{\displaystyle -zm_{2}(z)+\cdots+\frac{1}{-zm_{N}(z)+\frac{\displaystyle 1}{\displaystyle\tau(z)}}
    }}},
\end{equation}
 where the parameter $\tau$ satisfies the following
 \begin{equation}\label{eq:tau_j}
    \frac{\displaystyle1}{\displaystyle\tau(z)}=o(z),\quad
    z\widehat{\rightarrow}\infty.
\end{equation}
and atoms $(m_j, l_j)$ can be calculated by 
  \begin{equation}\label{eqimnm1:}\begin{split}&
       d_1=\cfrac{1}{b_0},\quad l_1=-\cfrac{1}{d_1a_0(0)},\quad
	d_j=\cfrac{1}{b_{j-1}\left(l_{j-1}\right)^2 d_{j-1}},\\&
	l_j=-\cfrac{l_{j-1}}{1+l_{j-1}d_ja_{j-1}(0)}\mbox{ and }
	m_j(z)=\cfrac{\left(a_{j-1}(z)-a_{j-1}(0)\right)d_j}{z},
\end{split}\end{equation}
where $d_j$ is the leading coefficient of the polynomial  $m_j$ and $\left(a_j,b_j\right)$ are defined by \eqref{14.int.th_3}.

Furthermore, \eqref{5p.eq:3.6*9} can be represented as 
 \begin{equation}\label{5p.eq:3.6*9x}
    f(z)=\frac{Q^{+}_{2N-1}(z)\tau(z)+Q^{+}_{2N-2}(z)}{P^{+}_{2N-1}(z)\tau(z)+P^{+}_{2N-2}(z)},
   \end{equation}
 where $P^{+}_{j}$ and $Q^{+}_j$ are called the Stieltjes polynomials of the first and second kind, respectively, which can be found as  solutions of the  system
 \begin{equation}\label{s4.3}
    \left\{
    \begin{array}{rcl}
    y_{2j}-y_{2j-2}=l_{j}(z)y_{2j-1},\\
    y_{2j+1}-y_{2j-1}=-zm_{j+1}(z)y_{2j},\\
    \end{array}\right.
\end{equation}
subject to the initial conditions
\begin{equation}\label{5p.eq:3.6*9xx}
   Q^{+}_{-1}(z)\equiv1, \quad Q^{+}_{0}(z)\equiv0; \qquad P^{+}_{-1}(z)\equiv0, \quad P^{+}_{0}(z)\equiv1.
\end{equation}

 \end{theorem}

\begin{theorem}\label{14p.int.th_ind**}  Let ${\mathbf
s}=\{s_i\}_{i=0}^{2n_N-1}$ be a regular sequence  and  let
$\cN({\mathbf s})=\{n_j\}_{j=1}^N$ be the set of normal indices of  ${\mathbf
s}$. Then any solution of the moment problem $MP(\mathbf{s}, 2n_N-1)$
 admits the following representation 
 \begin{equation}\label{5p.eq:f1}
    f(z)= \frac{1}{\displaystyle -z m_{1}(z)+\frac{1}{\displaystyle
    l_{1}(z)+\cdots+\frac{1}{-zm_{N}(z)+\displaystyle\frac{1}{\displaystyle
    l_{N}(z)+\tau(z)}    }}},
\end{equation}
where the parameter $\tau$ satisfies the following condition
 \begin{equation}\label{eq:tau_j009}
    \tau(z)=o(1),\quad  z\widehat{\rightarrow}\infty,
\end{equation}
and atoms $(m_j, l_j)$ can befound by \eqref{eqimnm1:}.
\end{theorem}

 \section{Truncated problem and  \textbf{P}-fractions}
 
 In the present section, we study a truncated multidimensional moment problem in terms of the Stieltjes transform.  In this case, we find a description of the solutions  in terms of  \textbf{P}-fractions.

 \begin{theorem}\label{20p.th.3.1} Let $\mathbf{s}=\{s_{i_1,i_2,...,i_n}\}_{i_1,...,i_n=0}^{2\mathfrak{n}-1}$ be a sequence of real numbers and let $\mathfrak{n}$ be a normal index for all associated sequences $\mathfrak s_{j_2,\ldots, j_n}=\{ \mathfrak{s}_{i,j_2,\ldots, j_n}\}_{i=0}^{2\mathfrak{n}-1}$ defined by~\eqref{21.1.13sw3}. Then
 any solution of the  problem $\bold{MP}(\mathbf{s}, 2\mathfrak{n}-1)$ takes the following representation

  \begin{equation}\label{20p.3.1}
   F(z_1,z_2,...,z_n)=  \sum_{j_2,\ldots, j_n=0}^{2\mathfrak{n}-1} \cfrac{1}{z_2^{j_2+1}\cdot\ldots\cdot z_n^{j_n+1}}  \underset{i=0}{\overset{N_{j_2,...,j_n}-1}{\mathbf{K}}}\left(-\cfrac{b_{i}^{(j_2,...,j_n)}}{a_{i}^{(j_2,...,j_n)}(z_1)}\right),
 \end{equation}
 where the parameters $\tau_{j_2,...,j_n}$ satisfy the following
\begin{equation}\label{20p.3.2}
	\tau_{j_2,...,j_n}(z_1)=o(1).
\end{equation}
 The atoms $\left(a^{(j_2,...,j_n)}_i,b^{(j_2,...,j_n)}_i\right)$ can be found by
 \begin{equation}\label{20p.3.3}\begin{split}
	&b^{(j_2,...,j_n)}_0=\mathfrak s^{(j_2,...,j_n)}_{n^{(j_2,...,j_n)}_1-1}\quad  \mbox{and}\quad b^{(j_2,...,j_n)}_j=\mathfrak s^{(j_2,...,j_n, j)}_{n^{(j_2,...,j_n)}_{j}-n^{(j_2,...,j_n)}_{j-1}-1}\\&
	 a^{(j_2,...,j_n)}_{j}(z_1)=\cfrac{1}{D^{(j_2,...,j_n,j-1)}_{\nu^{(j_2,...,j_n)}}} \begin{vmatrix}
	\mathfrak  s^{(j_2,...,j_n, j-1)}_{0}&\mathfrak  s^{(j_2,...,j_n, j-1)}_{1}& \ldots&  \mathfrak s^{(j_2,...,j_n, j-1)}_{\nu}\\
	  \ldots& \ldots&\ldots& \ldots\\
	   \mathfrak s^{(j_2,...,j_n, j-1)}_{\nu^{(j_2,...,j_n)}-1}& \mathfrak s^{(j_2,...,j_n, j-1)}_{\nu^{(j_2,...,j_n)}}& \ldots&\mathfrak  s^{(j_2,...,j_n, j-1)}_{2\nu^{(j_2,...,j_n)}-1}\\
	    1& z_1&\ldots& z^{\nu^{(j_2,...,j_n)}}_1
	 \end{vmatrix}
	 ,\\&
	\mathfrak s^{(j_2,...,j_n, j)}_i=\cfrac{(-1)^{i+\nu}}{\left(\mathfrak s^{(j_2,...,j_n, j-1)}_{\nu^{(j_2,...,j_n)}-1}\right)^{i+\nu^{(j_2,...,j_n)}+2}}\times\\&\quad\quad\quad\qquad\times\begin{vmatrix}
\mathfrak s^{(j_2,...,j_n, j-1)}_{\nu^{(j_2,...,j_n)}}& \mathfrak s^{(j_2,...,j_n, j-1)}_{\nu^{(j_2,...,j_n)}-1} &0&\ldots& 0\\
\vdots& \ddots &\ddots&\ddots & \vdots\\
\vdots&  &\ddots&\ddots &0\\
\vdots&  & &\ddots &\mathfrak  s^{(j_2,...,j_n, j-1)}_{\nu^{j_2,...,j_n}-1}\\
\mathfrak s^{(j_2,...,j_n, j-1)}_{2\nu^{(j_2,...,j_n)}+i}& \ldots &\ldots&\ldots& \mathfrak s^{(j_2,...,j_n, j-1)}_{\nu^{(j_2,...,j_n)}-1}
\end{vmatrix},
\end{split}
\end{equation}
where $i=\overline{0,2n^{(j_2,...,j_n)}_{N_{j_2,...,j_n}}-2n^{(j_2,...,j_n)}_j-1}$, $j=\overline {0, N_{j_2,...,j_n}}$, $\mathfrak s^{(j_2,...,j_n, 0)}_{\nu_{j_2,...,j_n}}=\mathfrak s^{(j_2,...,j_n)}_{\nu_{j_2,...,j_n}}$,  $\nu_{j_2,...,j_n}=n^{(j_2,...,j_n)}_j-n^{(j_2,...,j_n)}_{j-1}$ and $n^{(j_2,...,j_n)}_{0}=0$, $\mathcal{N}^{(j_2,...,j_n)}\left(\textbf{s}^{^{(j_2,...,j_n)}}\right)=\{n^{(j_2,...,j_n)}_{j}\}_{j=1}^{N_{j_2,...,j_n}}$ is the set of normal indices  of the sequence $\textbf{s}^{^{(j_2,...,j_n)}}$ and $n_{N_{j_2,...,j_n}}^{(j_2,...,j_n)}=\mathfrak{n}$.

 Furthermore,  \eqref{20p.3.1} can be rewritten in terms of  polynomials of the first and second kind as
  \begin{equation}\label{20p.3.4}
  	  F(z_1,...,z_n)=\!\! -\!\!\!\!\sum_{j_2,\ldots, j_n=0}^{2\mathfrak{n}-1} \cfrac{1}{\prod\limits_{i =2}^n\!z_i^{j_i+1} } 
	  \cfrac{Q^{({j_2,...,j_n)}}_{n_{N_{j_2,...,j_n}-1}}(z_1)\tau_{j_2,...,j_n}(z_1)+Q^{({j_2,...,j_n)}}_{n_{N_{j_2,...,j_n}}}(z_1)}{P^{({j_2,...,j_n)}}_{n_{N_{j_2,...,j_n}-1}}(z_1)\tau_{j_2,...,j_n}(z_1)+P^{({j_2,...,j_n)}}_{n_{N_{j_2,...,j_n}}}(z_1)},
  \end{equation}
  where $P^{({j_2,...,j_n)}}_{j}$ and $Q^{({j_2,...,j_n)}}_{j}$ are the polynomials of the first and second kind associated with the sequence $\mathfrak s_{j_2,\ldots, j_n}=\{ \mathfrak{s}_{i,j_2,\ldots, j_n}\}_{i=0}^{2\mathfrak{n}-1}$, respectively,   $P^{({j_2,...,j_n)}}_{j}$ and $Q^{({j_2,...,j_n)}}_{j}$ can be found as solutions of the following system
    \begin{equation}\label{20p.3.5}\begin{split}&
    	b_j^{(j_2,...,j_n)}y_{n_{j-1}}^{(j_2,...,j_n)}(z_1)-a_{j}^{(j_2,...,j_n)}(z_1)y_{n_j}^{(j_2,...,j_n)}(z_1)+y_{n_{j+1}}^{(j_2,...,j_n)}(z_1)=0,\\&
	P^{({j_2,...,j_n)}}_{-1}(z_1)\equiv0,\quad\mbox{and}\quad  P^{({j_2,...,j_n)}}_{0}(z_1)\equiv1,\\&
	Q^{({j_2,...,j_n)}}_{-1}(z_1)\equiv-1\quad\mbox{and}\quad  Q^{({j_2,...,j_n)}}_{0}(z_1)\equiv0.
 \end{split}  \end{equation}
 \end{theorem}

  \begin{proof} Let the assumptions of  theorem hold. Due to \eqref{20.2.3wq1},  any solution of  $\bold{MP}(\mathbf{s}, 2\mathfrak{n}-1)$  admits the following asymptotic expansion
  
\begin{equation}\label{20p.3.6xx}
   F(z_1,z_2,...,z_n)=  \sum_{j_2,\ldots, j_n=0}^{2\mathfrak{n}-1} \cfrac{1}{z_2^{j_2+1}\cdot\ldots\cdot z_n^{j_n+1}}F_{j_2,\ldots, j_n}(z_1),
\end{equation}
  where   all functions $F_{j_2,\ldots, j_n}$ depend on the variable $z_1$, $\mathfrak{n}=n_{N_{j_2,...,j_n}}^{(j_2,...,j_n)}$ and
  \[
  	F_{j_2,\ldots, j_n}(z_1)=-\cfrac{ \mathfrak{s}_{0,j_2,\ldots, j_n}}{z_1}-\cfrac{ \mathfrak{s}_{1,j_2,\ldots, j_n}}{z_1^2}-\ldots-\cfrac{ \mathfrak{s}_{2n_{N_{j_2,...,j_n}}^{(j_2,...,j_n)}-1,j_2,\ldots, j_n}}{z_1^{2n_{N_{j_2,...,j_n}}^{(j_2,...,j_n)}}}+o\left(\cfrac{1}{z_1^{2n_{N_{j_2,...,j_n}}^{(j_2,...,j_n)}}}\right).
  \]
  
  By Theorem~\ref{14p.int.th_ind},  $F_{j_2,\ldots, j_n}$ can be represented in terms of  \textbf{P}-fractions as
\begin{equation}\label{20p.3.6}
  	F_{j_2,\ldots, j_n}(z_1)=-\cfrac{b_{0}^{(j_2,\ldots, j_n)}}{a_0^{(j_2,\ldots, j_n)}(z_1)-\cfrac{b_1^{(j_2,\ldots, j_n)}}{a_1^{(j_2,\ldots, j_n)}-\ldots-\cfrac{b_{N_{j_2,...,j_n}-1}^{j_2,\ldots, j_n}}{a_{N_{j_2,...,j_n}-1}^{(j_2,\ldots, j_n)}(z_1)+\tau_{j_2,\ldots, j_n}(z_1)}}},
\end{equation}
where the parameters $\tau_{j_2,\ldots, j_n}$  satisfy \eqref{20p.3.2} and atoms $\left(a_j^{(j_2,\ldots, j_n)}, b_{j}^{(j_2,\ldots, j_n)}\right)$ are calculated by~\eqref{20p.3.3}. Moreover, \eqref{20p.3.6} can be rewritten in  short form as 
\[
F_{j_2,\ldots, j_n}(z_1)= \underset{i=0}{\overset{N_{j_2,...,j_n}-1}{\mathbf{K}}}\left(-\cfrac{b_{i}^{(j_2,...,j_n)}}{a_{i}^{(j_2,...,j_n)}(z_1)}\right).
\]

Hence
\[
	  F(z_1,z_2,...,z_n)= \!\! \sum_{j_2,\ldots, j_n=0}^{2\mathfrak{n}-1} \!\cfrac{F_{j_2,\ldots, j_n}(z_1)}{z_2^{j_2+1}\cdot\ldots\cdot z_n^{j_n+1}}=
	   \sum_{j_2,\ldots, j_n=0}^{\mathfrak{n}-1}\!\! \cfrac{ \underset{i=0}{\overset{N_{j_2,...,j_n}-1}{\mathbf{K}}}\left(-\cfrac{b_{i}^{(j_2,...,j_n)}}{a_{i}^{(j_2,...,j_n)}(z_1)}\right)}{z_2^{j_2+1}\cdot\ldots\cdot z_n^{j_n+1}} 
\]
and \eqref{20p.3.1} is proved. 

Furthermore, by Theorem~\ref{14p.int.th_ind}, \eqref{20p.3.6} can be rewritten in terms of the polynomials of the first and second  kind as
\begin{equation}\label{20p.3.6xxx}
\begin{split}
	F_{j_2,\ldots, j_n}(z_1)&= \underset{i=0}{\overset{N_{j_2,...,j_n}-1}{\mathbf{K}}}\left(-\cfrac{b_{i}^{(j_2,...,j_n)}}{a_{i}^{(j_2,...,j_n)}(z_1)}\right)=\\&=
	 - \cfrac{Q^{({j_2,...,j_n)}}_{n_{N_{j_2,...,j_n}-1}}(z_1)\tau_{j_2,...,j_n}(z_1)+Q^{({j_2,...,j_n)}}_{n_{N_{j_2,...,j_n}}}(z_1)}{P^{({j_2,...,j_n)}}_{n_{N_{j_2,...,j_n}-1}}(z_1)\tau_{j_2,...,j_n}(z_1)+P^{({j_2,...,j_n)}}_{n_{N_{j_2,...,j_n}}}(z_1)},
\end{split}\end{equation}
where $P^{({j_2,...,j_n)}}_{j}$ and $Q^{({j_2,...,j_n)}}_{j}$ are the polynomials of the first and second kind, respectively, which are associated with the sequence $\mathfrak s_{j_2,\ldots, j_n}=\{ \mathfrak{s}_{i,j_2,\ldots, j_n}\}_{i=0}^{2\mathfrak{n}-1}$. Moreover, $P^{({j_2,...,j_n)}}_{j}$ and $Q^{({j_2,...,j_n)}}_{j}$   can be defined as the solutions of the system~\eqref{20p.3.5}. Substituting \eqref{20p.3.6xxx} into \eqref{20p.3.6xx}, we obtain \eqref{20p.3.4}. This completes the proof.~\end{proof}

 \begin{corollary}\label{20p.cor.3.2} Let $\mu$ be a nonnegative Borel measure on $\mathbb{R}_{+}^n$, where $\supp(\mu)=A\subseteq \mathbb{R}_{+}^n$ and let a moment sequence 
 $\mathbf{s}=\{s_{i_1,\ldots,i_n}\}_{i_1,\ldots, i_n=0}^{2\mathfrak{n}-1}$ be defined by \eqref{14.2.1}. Then
\begin{equation}\label{20p.3.9}\small
 \int_{\mathbb{R}_{+}^n}\cfrac{d\mu(t_1,t_2,..., t_n)}{1- \sum\limits_{i=1}^n \cfrac{t_i}{z_i}}=\!\!\!\!\sum_{j_2,\ldots, j_n=0}^{\mathfrak{n}-1} \cfrac{z_1}{\prod\limits_{i =2}^n\!z_i^{j_i} } 
	  \cfrac{Q^{({j_2,...,j_n)}}_{n_{N_{j_2,...,j_n}-1}}(z_1)\tau_{j_2,...,j_n}(z_1)+Q^{({j_2,...,j_n)}}_{n_{N_{j_2,...,j_n}}}(z_1)}{P^{({j_2,...,j_n)}}_{n_{N_{j_2,...,j_n}-1}}(z_1)\tau_{j_2,...,j_n}(z_1)+P^{({j_2,...,j_n)}}_{n_{N_{j_2,...,j_n}}}(z_1)},
 \end{equation}
 where  the parameters $\tau_{j_2,...,j_n}$ satisfy \eqref{20p.3.2},  $P^{({j_2,...,j_n)}}_{n_j^{({j_2,...,j_n)}}}$ and $Q^{({j_2,...,j_n)}}_{n_j^{({j_2,...,j_n)}}}$ are the polynomials of the first and second kind associated with the sequence $\mathfrak s_{j_2,\ldots, j_n}=\{ \mathfrak{s}_{i,j_2,\ldots, j_n}\}_{i=0}^{2\mathfrak{n}-1}$, which can be defined as the solutions of  \eqref{20p.3.5}.
 \end{corollary}
 
 \begin{proof} Let the assumptions of  statement hold. By \eqref{20.2.2} and Theorem~\ref{20p.th.3.1},  the associated function $F$ admits the following representation
 \[\small\begin{split}
 	 F(z_1,z_2,...,z_n)&=-\frac{1}{\prod\limits_{i=1}^n z_i}\int_{\mathbb{R}_{+}^n}\cfrac{d\mu(t_1,t_2,..., t_n)}{1- \sum\limits_{i=1}^n \cfrac{t_i}{z_i}}=\\&=
	 -\!\!\!\!\sum_{j_2,\ldots, j_n=0}^{\mathfrak{n}-1} \cfrac{1}{\prod\limits_{i =2}^n\!z_i^{j_i+1} } 
	  \cfrac{Q^{({j_2,...,j_n)}}_{n_{N_{j_2,...,j_n}-1}}(z_1)\tau_{j_2,...,j_n}(z_1)+Q^{({j_2,...,j_n)}}_{n_{N_{j_2,...,j_n}}}(z_1)}{P^{({j_2,...,j_n)}}_{n_{N_{j_2,...,j_n}-1}}(z_1)\tau_{j_2,...,j_n}(z_1)+P^{({j_2,...,j_n)}}_{n_{N_{j_2,...,j_n}}}(z_1)}.
 \end{split}\]
Consequently,  we obtain
\[\small\begin{split}
	\int_{\mathbb{R}_{+}^n}\!\!\cfrac{d\mu(t_1,\!t_2,...,\! t_n)}{1- \sum\limits_{i=1}^n \cfrac{t_i}{z_i}}&\!= \!\!\!\!\sum_{j_2,\ldots, j_n=0}^{\mathfrak{n}-1} \cfrac{\prod\limits_{i=1}^n z_i}{\prod\limits_{i =2}^n\!z_i^{j_i+1}} 
	  \cfrac{Q^{({j_2,...,j_n)}}_{n_{N_{j_2,...,j_n}-1}}(z_1)\tau_{j_2,...,j_n}(z_1)\!+\!Q^{({j_2,...,j_n)}}_{n_{N_{j_2,...,j_n}}}(z_1)}{P^{({j_2,...,j_n)}}_{n_{N_{j_2,...,j_n}-1}}(z_1)\tau_{j_2,...,j_n}(z_1)\!+\!P^{({j_2,...,j_n)}}_{n_{N_{j_2,...,j_n}}}(z_1)}\!=\!\\&=\!\!\!\!\sum_{j_2,\ldots, j_n=0}^{\mathfrak{n}-1} \cfrac{z_1}{\prod\limits_{i =2}^n\!z_i^{j_i}} 
	  \cfrac{Q^{({j_2,...,j_n)}}_{n_{N_{j_2,...,j_n}-1}}(z_1)\tau_{j_2,...,j_n}(z_1)+Q^{({j_2,...,j_n)}}_{n_{N_{j_2,...,j_n}}}(z_1)}{P^{({j_2,...,j_n)}}_{n_{N_{j_2,...,j_n}-1}}(z_1)\tau_{j_2,...,j_n}(z_1)+P^{({j_2,...,j_n)}}_{n_{N_{j_2,...,j_n}}}(z_1)}.
\end{split}\]
This completes the proof.~\end{proof}
  \section{Truncated problem and  \textbf{S}-fractions}
 
 In this section, we study the odd and even truncated multidimensional moment problem. The descriptions of these problems can be obtained in terms of  \textbf{S}-fractions and Stieltjes polynomials.
 
 \subsection{Regular case}
 
 First of all, we study the case, when all associated sequences $\mathfrak s_{j_2,\ldots, j_n}$ are regular.

  \begin{theorem}\label{20p.th.4.1}(odd case) Let $\mathbf{s}=\{s_{i_1,i_2,...,i_n}\}_{i_1,...,i_n=0}^{2\mathfrak{n}-2}$ be a sequence of real numbers and let $\mathfrak{n}$ be a normal index for all associated sequences $\mathfrak s_{j_2,\ldots, j_n}=\{ \mathfrak{s}_{i,j_2,\ldots, j_n}\}_{i=0}^{2\mathfrak{n}-2}$ defined by~\eqref{21.1.13sw3}. Let all sequences $\mathfrak s_{j_2,\ldots, j_n}$ be regular and let $\mathcal{N}^{(j_2,...,j_n)}\left(\textbf{s}^{^{(j_2,...,j_n)}}\right)=\{n^{(j_2,...,j_n)}_{j}\}_{j=1}^{N_{j_2,...,j_n}}$ be the set of normal indices  of $\textbf{s}^{^{(j_2,...,j_n)}}$ and $n_{N_{j_2,...,j_n}}^{(j_2,...,j_n)}=\mathfrak{n}$. Then
 any solution of the  problem $\bold{MP}(\mathbf{s}, 2\mathfrak{n}-2)$ takes the representation

  \begin{equation}\label{20p.4.1}\begin{split}
   &F(z_1,z_2,...,z_n)=  \sum_{j_2,\ldots, j_n=0}^{2\mathfrak{n}-2} \cfrac{1}{ \prod\limits_{i =2}^n\!z_i^{j_i+1}} \!\times\\&\!\times\!
      \frac{1}{\displaystyle\! -z_1 m^{(j_2,\ldots, j_n)}_{1}(z_1)\!+\!\frac{1}{\displaystyle
    l^{(j_2,\ldots, j_n)}_{1}\!+\!\cdots\!+\!\frac{1}{\!-z_1m^{(j_2,\ldots, j_n)}_{N_{j_2,...,j_n}}(z_1)\!+\!\frac{\displaystyle 1}{\displaystyle\tau_{j_2,\ldots, j_n}(z_1)}}
    }},
   \end{split}
 \end{equation}
  where the parameter $\tau_{j_2,\ldots, j_n}$ satisfies the following
 \begin{equation}\label{20p.4.2}
    \frac{\displaystyle1}{\displaystyle\tau_{j_2,\ldots, j_n}(z_1)}=o(z_1)
    \end{equation}
and atoms $\left(m^{(j_2,\ldots, j_n)}_{j}, l^{(j_2,\ldots, j_n)}_{j}\right)$ can be found   by
 \begin{equation}\label{20p.4.3}\begin{split}&
	d_1^{(j_2,\ldots, j_n)}=\cfrac{1}{b_0^{(j_2,\ldots, j_n)}},\quad l_1^{(j_2,\ldots, j_n)}=-\cfrac{1}{d_1^{(j_2,\ldots, j_n)}a_0^{(j_2,\ldots, j_n)}(0)},\\&
	d_j^{(j_2,\ldots, j_n)}=\cfrac{1}{b_{j-1}^{(j_2,\ldots, j_n)}\left(l_{j-1}^{(j_2,\ldots, j_n)}\right)^2 d_{j-1}^{(j_2,\ldots, j_n)}},\\&
	l_j^{(j_2,\ldots, j_n)}=-\cfrac{l_{j-1}^{(j_2,\ldots, j_n)}}{1+l_{j-1}^{(j_2,\ldots, j_n)}d_j^{(j_2,\ldots, j_n)}a_{j-1}^{(j_2,\ldots, j_n)}(0)},\\&
	m_j^{(j_2,\ldots, j_n)}(z_1)=\cfrac{\left(a_{j-1}^{(j_2,\ldots, j_n)}(z_1)-a_{j-1}^{(j_2,\ldots, j_n)}(0)\right)d_j^{(j_2,\ldots, j_n)}}{z_1},
\end{split}\end{equation}
$d_j^{(j_2,\ldots, j_n)}$ is the leading coefficient of the polynomial  $m_j^{(j_2,\ldots, j_n)}$ and $\left(a^{(j_2,...,j_n)}_j,b^{(j_2,...,j_n)}_j\right)$ are defined by \eqref{20p.3.3}.

Furthermore, \eqref{20p.4.1} can be rewritten in terms of the Stieltjes polynomials as
\begin{equation}\label{20p.4.4}
	F(z_1,z_2,...,z_n)\!=\!\! \!\! \sum_{j_2,\ldots, j_n=0}^{2\mathfrak{n}-2} \cfrac{1}{ \prod\limits_{i =2}^n\!\!z_i^{j_i+1}} \cfrac{Q^{+(j_2,\ldots, j_n)}_{2N_{j_2,\ldots, j_n}-1}(z_1)\tau_{j_2,\ldots, j_n}(z_1)\!+\!Q^{+(j_2,\ldots, j_n)}_{2N_{j_2,\ldots, j_n}-2}(z_1)}{P^{+(j_2,\ldots, j_n)}_{2N_{j_2,\ldots, j_n}-1}(z_1)\tau_{j_2,\ldots, j_n}(z_1)\!+\!P^{+(j_2,\ldots, j_n)}_{2N_{j_2,\ldots, j_n}-2}(z_1)},
\end{equation}
where $P^{+(j_2,\ldots, j_n)}_{j}$ and $Q^{+(j_2,\ldots, j_n)}_{j}$ are the Stieltjes polynomials of the first and second kind associated with the sequence  $\mathfrak s_{j_2,\ldots, j_n}=\{ \mathfrak{s}_{i,j_2,\ldots, j_n}\}_{i=0}^{2\mathfrak{n}-2}$, respectively, $P^{+(j_2,\ldots, j_n)}_{j}$. and $Q^{+(j_2,\ldots, j_n)}_{j}$ can be found as the solutions of the following system 
\begin{equation}\label{20p.4.5}
  \left\{
    \begin{array}{rcl}
    y_{2j}^{(j_2,\ldots, j_n)}-y^{(j_2,\ldots, j_n)}_{2j-2}=l^{(j_2,\ldots, j_n)}_{j}(z_1)y^{(j_2,\ldots, j_n)}_{2j-1}\,,\\
    y^{(j_2,\ldots, j_n)}_{2j+1}-y^{(j_2,\ldots, j_n)}_{2j-1}=-z_1m^{(j_2,\ldots, j_n)}_{j+1}(z_1)y^{(j_2,\ldots, j_n)}_{2j}\\
    \end{array}\right.
\end{equation}
subject to the initial conditions
\begin{equation}\label{20p.4.6}\begin{split}&
   Q^{+(j_2,\ldots, j_n)}_{-1}(z_1)\equiv1\quad\mbox{and}\quad Q^{+(j_2,\ldots, j_n)}_{0}(z_1)\equiv0, \\&P^{+(j_2,\ldots, j_n)}_{-1}(z_1)\equiv0\quad\mbox{and}\quad P^{+(j_2,\ldots, j_n)}_{0}(z_1)\equiv1.
\end{split}\end{equation}
 \end{theorem}
\begin{proof}
Let the assumptions of  theorem hold.   Any solution of  $\bold{MP}(\mathbf{s}, 2\mathfrak{n}-2)$  admits  the expansion 
\begin{equation}\label{20p.4.7}
   F(z_1,z_2,...,z_n)=  \sum_{j_2,\ldots, j_n=0}^{2\mathfrak{n}-2} \cfrac{1}{z_2^{j_2+1}\cdot\ldots\cdot z_n^{j_n+1}}F_{j_2,\ldots, j_n}(z_1),
\end{equation}
  where   $F_{j_2,\ldots, j_n}$ depends on the variable $z_1$ and $\mathfrak{n}=n_{N_{j_2,...,j_n}}$, we obtain 
  \[
  	F_{j_2,\ldots, j_n}(z_1)=-\cfrac{ \mathfrak{s}_{0,j_2,\ldots, j_n}}{z_1}-\ldots-\cfrac{ \mathfrak{s}_{2n_{N_{j_2,...,j_n}}-2,j_2,\ldots, j_n}}{z_1^{2n_{N_{j_2,...,j_n}}-1}}+o\left(\cfrac{1}{z_1^{2n_{N_{j_2,...,j_n}}-1}}\right).
  \]

By Theorem~\ref{2p.pr.alg1},  $F_{j_2,\ldots, j_n}$ can be represented in term of $\textbf{S}$-fraction as follows
\begin{equation}\label{20p.4.8}\small
	F_{j_2,\ldots, j_n}(z_1)= \frac{1}{\displaystyle\! -z_1 m^{(j_2,\ldots, j_n)}_{1}(z_1)\!+\!\frac{1}{\displaystyle
    l^{(j_2,\ldots, j_n)}_{1}\!+\!\cdots\!+\!\frac{1}{\!-z_1m^{(j_2,\ldots, j_n)}_{N_{j_2,...,j_n}}(z_1)\!+\!\frac{\displaystyle 1}{\displaystyle\tau_{j_2,\ldots, j_n}(z_1)}}
    }},
\end{equation}
where the parameter $\tau_{j_2,...,j_n}$ satisfies \eqref{20p.4.2} and  atoms $\left(m^{(j_2,\ldots, j_n)}_{j}, l^{(j_2,\ldots, j_n)}_{j}\right)$ are calculated by~\eqref{20p.4.3}.
Combining \eqref{20p.4.7} and \eqref{20p.4.8}, we obtain~\eqref{20p.4.1}.

On the other hand,  by Theorem~\ref{2p.pr.alg1} (see \eqref{5p.eq:3.6*9x}--\eqref{5p.eq:3.6*9xx}),  \eqref{20p.4.8} can be rewritten in terms of the Stieltjes polynomials as
\begin{equation}\label{20p.4.9}
	F_{j_2,\ldots, j_n}(z_1)=\cfrac{Q^{+(j_2,\ldots, j_n)}_{2N_{j_2,\ldots, j_n}-1}(z_1)\tau_{j_2,\ldots, j_n}(z_1)+Q^{+(j_2,\ldots, j_n)}_{2N_{j_2,\ldots, j_n}-2}(z_1)}{P^{+(j_2,\ldots, j_n)}_{2N_{j_2,\ldots, j_n}-1}(z_1)\tau_{j_2,\ldots, j_n}(z_1)+P^{+(j_2,\ldots, j_n)}_{2N_{j_2,\ldots, j_n}-2}(z_1)},
\end{equation}
where $P^{+(j_2,\ldots, j_n)}_{j}$ and $Q^{+(j_2,\ldots, j_n)}_{j}$ are the solutions of the  system~\eqref{20p.4.5}--\eqref{20p.4.6}. Substituting~\eqref{20p.4.9} into \eqref{20p.4.7}, we obtain \eqref{20p.4.4}. This completes the proof.~\end{proof}

 \begin{corollary}\label{20p.cor.4.2} Let $\mu$ be a nonnegative Borel measure on $\mathbb{R}_{+}^n$, where $\supp(\mu)=A\subseteq \mathbb{R}_{+}^n$ and let a moment sequence 
 $\mathbf{s}=\{s_{i_1,\ldots,i_n}\}_{i_1,\ldots, i_n=0}^{2\mathfrak{n}-2}$ be generated by \eqref{14.2.1}. Then
\begin{equation}\label{20p.4.10!}\small
 \int_{\mathbb{R}_{+}^n}\cfrac{d\mu(t_1,..., t_n)}{1- \sum\limits_{i=1}^n \cfrac{t_i}{z_i}}=  -\!\! \!\! \sum_{j_2,\ldots, j_n=0}^{2\mathfrak{n}-2}  \cfrac{z_1}{\prod\limits_{i =2}^n\!z_i^{j_i} } \cfrac{Q^{+(j_2,\ldots, j_n)}_{2N_{j_2,\ldots, j_n}-1}(z_1)\tau_{j_2,\ldots, j_n}(z_1)\!+\!Q^{+(j_2,\ldots, j_n)}_{2N_{j_2,\ldots, j_n}-2}(z_1)}{P^{+(j_2,\ldots, j_n)}_{2N_{j_2,\ldots, j_n}-1}(z_1)\tau_{j_2,\ldots, j_n}(z_1)\!+\!P^{+(j_2,\ldots, j_n)}_{2N_{j_2,\ldots, j_n}-2}(z_1)},
 \end{equation}	  
  where  $\tau_{j_2,...,j_n}$ satisfies \eqref{20p.4.2} ,  $P^{+(j_2,\ldots, j_n)}_{j}$ and $Q^{+(j_2,\ldots, j_n)}_{j}$ are calculated by~\eqref{20p.4.5}--\eqref{20p.4.6}.
  \end{corollary}

 \begin{proof}  By \eqref{20.2.2} and Theorem~\ref{20p.th.4.1},  the associated function $F$ admits the following representation
 \[\small\begin{split}
 	 F(z_1,z_2,...,z_n)&=-\frac{1}{\prod\limits_{i=1}^n z_i}\int_{\mathbb{R}_{+}^n}\cfrac{d\mu(t_1,t_2,..., t_n)}{1- \sum\limits_{i=1}^n \cfrac{t_i}{z_i}}=\\&=
	\!\! \!\! \sum_{j_2,\ldots, j_n=0}^{2\mathfrak{n}-2} \cfrac{1}{ \prod\limits_{i =2}^n\!\!z_i^{j_i+1}} \cfrac{Q^{+(j_2,\ldots, j_n)}_{2N_{j_2,\ldots, j_n}-1}(z_1)\tau_{j_2,\ldots, j_n}(z_1)\!+\!Q^{+(j_2,\ldots, j_n)}_{2N_{j_2,\ldots, j_n}-2}(z_1)}{P^{+(j_2,\ldots, j_n)}_{2N_{j_2,\ldots, j_n}-1}(z_1)\tau_{j_2,\ldots, j_n}(z_1)\!+\!P^{+(j_2,\ldots, j_n)}_{2N_{j_2,\ldots, j_n}-2}(z_1)}, \end{split}\]
 where the parameters  $\tau_{j_2,...,j_n}$ satisfy \eqref{20p.4.2} ,  $P^{+(j_2,\ldots, j_n)}_{j}$ and $Q^{+(j_2,\ldots, j_n)}_{j}$ are solutions of the system~\eqref{20p.4.5}--\eqref{20p.4.6}.

Consequently,  we obtain
\[\small\begin{split}
	\int_{\mathbb{R}_{+}^n}\!\!\cfrac{d\mu(t_1,\!t_2,...,\! t_n)}{1- \sum\limits_{i=1}^n \cfrac{t_i}{z_i}}&\!=- \!\!\!\!\sum_{j_2,\ldots, j_n=0}^{2\mathfrak{n}-2} \cfrac{\prod\limits_{i=1}^n z_i}{\prod\limits_{i =2}^n\!z_i^{j_i+1}} 
	   \cfrac{Q^{+(j_2,\ldots, j_n)}_{2N_{j_2,\ldots, j_n}-1}(z_1)\tau_{j_2,\ldots, j_n}(z_1)\!+\!Q^{+(j_2,\ldots, j_n)}_{2N_{j_2,\ldots, j_n}-2}(z_1)}{P^{+(j_2,\ldots, j_n)}_{2N_{j_2,\ldots, j_n}-1}(z_1)\tau_{j_2,\ldots, j_n}(z_1)\!+\!P^{+(j_2,\ldots, j_n)}_{2N_{j_2,\ldots, j_n}-2}(z_1)}\!=\!\\&=-\!\! \!\! \sum_{j_2,\ldots, j_n=0}^{2\mathfrak{n}-2}  \cfrac{z_1}{\prod\limits_{i =2}^n\!z_i^{j_i} } \cfrac{Q^{+(j_2,\ldots, j_n)}_{2N_{j_2,\ldots, j_n}-1}(z_1)\tau_{j_2,\ldots, j_n}(z_1)\!+\!Q^{+(j_2,\ldots, j_n)}_{2N_{j_2,\ldots, j_n}-2}(z_1)}{P^{+(j_2,\ldots, j_n)}_{2N_{j_2,\ldots, j_n}-1}(z_1)\tau_{j_2,\ldots, j_n}(z_1)\!+\!P^{+(j_2,\ldots, j_n)}_{2N_{j_2,\ldots, j_n}-2}(z_1)}.
\end{split}\]
This completes the proof.~\end{proof}

 \begin{remark}\label{20p.rem4.2}(\cite{DK17})  Stieltjes polynomials $P^{+(j_2,\ldots, j_n)}_{j}$ and $Q^{+(j_2,\ldots, j_n)}_{j}$  can be found using alternative formulas
 \begin{equation}\label{20p.4.10}
    \begin{split}
        &P^{+(j_2,\ldots, j_n)}_{-1}(z_1)\equiv0,\, P^{+(j_2,\ldots, j_n)}_{0}(z_1)\equiv1,\,\\&
        Q^{+(j_2,\ldots, j_n)}_{-1}(z_1)\equiv1,\, Q^{+(j_2,\ldots, j_n)}_{0}(z_1)\equiv0,
        \\
        &P^{+(j_2,\ldots, j_n)}_{2i-1}(z_1)=\frac{-1}{b^{(j_2,\ldots, j_n)}_{0}\ldots b^{(j_2,\ldots, j_n)}_{i-1}}
        \begin{vmatrix}
            P^{(j_2,\ldots, j_n)}_{n_{i}}(z_1) & P^{(j_2,\ldots, j_n)}_{n_{i-1}}(z_1) \\
            P^{(j_2,\ldots, j_n)}_{n_{i}}(0) & P^{(j_2,\ldots, j_n)}_{n_{i-1}}(0)\\
        \end{vmatrix},\\&
        Q_{2i-1}^{+(j_2,\ldots, j_n)}(z_1)=\frac{1}{b^{(j_2,\ldots, j_n)}_{0}\ldots b^{(j_2,\ldots, j_n)}_{i-1}}
        \begin{vmatrix}
            Q^{(j_2,\ldots, j_n)}_{n_{i}}(z_1) & Q^{(j_2,\ldots, j_n)}_{n_{i-1}}(z_1) \\
            P^{(j_2,\ldots, j_n)}_{n_{i}}(0) & P^{(j_2,\ldots, j_n)}_{n_{i-1}}(0)\\
        \end{vmatrix},\\&  P^{+(j_2,\ldots, j_n)}_{2i}(z_1)=\frac{P^{(j_2,\ldots, j_n)}_{n_{i}}(z_1)}{P^{(j_2,\ldots, j_n)}_{n_{i}}(0)},\,
        Q^{+(j_2,\ldots, j_n)}_{2i}(z_1)=-\frac{Q^{(j_2,\ldots, j_n)}_{n_{i}}(z_1)}{P^{(j_2,\ldots, j_n)}_{n_{i}}(0)}.
    \end{split}
\end{equation}
 \end{remark}
 
 
   \begin{theorem}\label{20p.th.4.3} (even case)Let $\mathbf{s}=\{s_{i_1,i_2,...,i_n}\}_{i_1,...,i_n=0}^{2\mathfrak{n}-1}$ be a sequence of real numbers and let $\mathfrak{n}$ be a normal index for all associated sequences $\mathfrak s_{j_2,\ldots, j_n}=\{ \mathfrak{s}_{i,j_2,\ldots, j_n}\}_{i=0}^{2\mathfrak{n}-1}$ defined by~\eqref{21.1.13sw3}. Let all sequences $\mathfrak s_{j_2,\ldots, j_n}$ be regular and let $\mathcal{N}^{(j_2,...,j_n)}\left(\textbf{s}^{^{(j_2,...,j_n)}}\right)=\{n^{(j_2,...,j_n)}_{j}\}_{j=1}^{N_{j_2,...,j_n}}$ be the set of normal indices  of $\textbf{s}^{^{(j_2,...,j_n)}}$ and $n_{N_{j_2,...,j_n}}^{(j_2,...,j_n)}=\mathfrak{n}$. Then
 any solution of the  problem $\bold{MP}(\mathbf{s}, 2\mathfrak{n}-1)$ admits the representation

  \begin{equation}\label{20p.4.11}\begin{split}
   &F(z_1,z_2,...,z_n)=  \sum_{j_2,\ldots, j_n=0}^{2\mathfrak{n}-1} \cfrac{1}{ \prod\limits_{i =2}^n\!z_i^{j_i+1}} \!\times\\&\!\times\!
      \frac{1}{\displaystyle\! -z_1 m^{(j_2,\ldots, j_n)}_{1}(z_1)\!+\!\frac{1}{\displaystyle
    l^{(j_2,\ldots, j_n)}_{1}\!+\!\cdots\!+\!\frac{1}{\!l^{(j_2,\ldots, j_n)}_{N_{j_2,...,j_n}}(z_1)\!+\!\tau_{j_2,\ldots, j_n}(z_1)}
    }},
   \end{split}
 \end{equation}
  where the parameter $\tau_{j_2,\ldots, j_n}$ satisfies the following condition
 \begin{equation}\label{20p.4.12}
    \tau_{j_2,\ldots, j_n}(z_1)=o(1)
    \end{equation}
and atoms $\left(m^{(j_2,\ldots, j_n)}_{j}, l^{(j_2,\ldots, j_n)}_{j}\right)$ can be found   by~\eqref{20p.4.3}.
 
 Furthermore,  \eqref{20p.4.11} can be rewritten in terms of the Sieltjes polynomials as follows
 \begin{equation}\label{20p.4.13}
	F(z_1,z_2,...,z_n)\!=\!\! \!\! \sum_{j_2,\ldots, j_n=0}^{2\mathfrak{n}-1} \cfrac{1}{ \prod\limits_{i =2}^n\!\!z_i^{j_i+1}} \cfrac{Q^{+(j_2,\ldots, j_n)}_{2N_{j_2,\ldots, j_n}-1}(z_1)\tau_{j_2,\ldots, j_n}(z_1)\!+\!Q^{+(j_2,\ldots, j_n)}_{2N_{j_2,\ldots, j_n}}(z_1)}{P^{+(j_2,\ldots, j_n)}_{2N_{j_2,\ldots, j_n}-1}(z_1)\tau_{j_2,\ldots, j_n}(z_1)\!+\!P^{+(j_2,\ldots, j_n)}_{2N_{j_2,\ldots, j_n}}(z_1)},
\end{equation}
where the Stieltjes polynomials $P^{+(j_2,\ldots, j_n)}_{j}$ and $Q^{+(j_2,\ldots, j_n)}_{j}$ are associated with the sequence $\mathfrak s_{j_2,\ldots, j_n}=\{ \mathfrak{s}_{i,j_2,\ldots, j_n}\}_{i=0}^{2n_{N_{j_2,...,j_n}}-1}$ .
 \end{theorem}

 \begin{proof} Let the assumptions of theorem hold, Hence, any solution of the  problem $\bold{MP}(\mathbf{s}, 2\mathfrak{n}-1)$ admits the following asymptotic expansion 
 \begin{equation}\label{20p.4.14}
   F(z_1,z_2,...,z_n)=  \sum_{j_2,\ldots, j_n=0}^{2\mathfrak{n}-1} \cfrac{1}{z_2^{j_2+1}\cdot\ldots\cdot z_n^{j_n+1}}F_{j_2,\ldots, j_n}(z_1),
\end{equation}
  where   $F_{j_2,\ldots, j_n}$ depends on the variable $z_1$ and $\mathfrak{n}=n_{N_{j_2,...,j_n}}$, we get
  \[
  	F_{j_2,\ldots, j_n}(z_1)=-\cfrac{ \mathfrak{s}_{0,j_2,\ldots, j_n}}{z_1}-\ldots-\cfrac{ \mathfrak{s}_{2n_{N_{j_2,...,j_n}}-1,j_2,\ldots, j_n}}{z_1^{2n_{N_{j_2,...,j_n}}}}+o\left(\cfrac{1}{z_1^{2n_{N_{j_2,...,j_n}}}}\right).
  \]

By Theorem~\ref{14p.int.th_ind**} , we obtain
  \begin{equation}\label{20p.4.15}
  F_{j_2,\ldots, j_n}(z_1)=     \frac{1}{\displaystyle\! -z_1 m^{(j_2,\ldots, j_n)}_{1}(z_1)\!+\!\frac{1}{\displaystyle
    l^{(j_2,\ldots, j_n)}_{1}\!+\!\cdots\!+\!\frac{1}{\!l^{(j_2,\ldots, j_n)}_{N_{j_2,...,j_n}}(z_1)\!+\!\tau_{j_2,\ldots, j_n}(z_1)}
    }},
  \end{equation}
where the atoms $\left(m^{(j_2,\ldots, j_n)}_{j}, l^{(j_2,\ldots, j_n)}_{j}\right)$ are calculated by~\eqref{20p.4.3} and parameter $\tau_{j_2,\ldots, j_n}$ satisfies~\eqref{20p.4.12}. 
   Substituting~\eqref{20p.4.15} into \eqref{20p.4.14}, we get \eqref{20p.4.11}. 
   
On the other hand, $F_{j_2,\ldots, j_n}$ can be rewritten in terns of the Stiejtjes polynomials and by Theorem~\ref{14p.int.th_ind**}, we obtain
   \begin{equation}\label{20p.4.16}
  F_{j_2,\ldots, j_n}(z_1)=    \cfrac{Q^{+(j_2,\ldots, j_n)}_{2N_{j_2,\ldots, j_n}-1}(z_1)\tau_{j_2,\ldots, j_n}(z_1)\!+\!Q^{+(j_2,\ldots, j_n)}_{2N_{j_2,\ldots, j_n}}(z_1)}{P^{+(j_2,\ldots, j_n)}_{2N_{j_2,\ldots, j_n}-1}(z_1)\tau_{j_2,\ldots, j_n}(z_1)\!+\!P^{+(j_2,\ldots, j_n)}_{2N_{j_2,\ldots, j_n}}(z_1)},
    \end{equation}
  where  $P^{+(j_2,\ldots, j_n)}_{j}$ and $Q^{+(j_2,\ldots, j_n)}_{j}$  are the Stieltjes polynomials of the first and second kind, respectively, which are associated with the sequence $\mathfrak s_{j_2,\ldots, j_n}=\{ \mathfrak{s}_{i,j_2,\ldots, j_n}\}_{i=0}^{2\mathfrak{n}-1}$ and calculated by~\eqref{20p.4.5}--\eqref{20p.4.6} (or \eqref{20p.4.10}). Combining \eqref{20p.4.14} and \eqref{20p.4.16}, we obtain~\eqref{20p.4.11}. This completes the proof.~\end{proof}
  \begin{corollary}\label{20p.cor.4.5} Let $\mu$ be a nonnegative Borel measure on $\mathbb{R}_{+}^n$, where $\supp(\mu)=A\subseteq \mathbb{R}_{+}^n$ and let a moment sequence 
 $\mathbf{s}=\{s_{i_1,\ldots,i_n}\}_{i_1,\ldots, i_n=0}^{2\mathfrak{n}-1}$ be generated by \eqref{14.2.1}. Then
\begin{equation}\label{20p.4.10!!}\small
 \int_{\mathbb{R}_{+}^n}\cfrac{d\mu(t_1,..., t_n)}{1- \sum\limits_{i=1}^n \cfrac{t_i}{z_i}}=  -\!\! \!\! \sum_{j_2,\ldots, j_n=0}^{2\mathfrak{n}-1}  \cfrac{z_1}{\prod\limits_{i =2}^n\!z_i^{j_i} } \cfrac{Q^{+(j_2,\ldots, j_n)}_{2N_{j_2,\ldots, j_n}-1}(z_1)\tau_{j_2,\ldots, j_n}(z_1)\!+\!Q^{+(j_2,\ldots, j_n)}_{2N_{j_2,\ldots, j_n}}(z_1)}{P^{+(j_2,\ldots, j_n)}_{2N_{j_2,\ldots, j_n}-1}(z_1)\tau_{j_2,\ldots, j_n}(z_1)\!+\!P^{+(j_2,\ldots, j_n)}_{2N_{j_2,\ldots, j_n}}(z_1)},
 \end{equation}	  
  where  $\tau_{j_2,...,j_n}$ satisfies \eqref{20p.4.12} ,  $P^{+(j_2,\ldots, j_n)}_{j}$ and $Q^{+(j_2,\ldots, j_n)}_{j}$ are calculated by~\eqref{20p.4.5}--\eqref{20p.4.6}.
  \end{corollary}

 \subsection{$\alpha$-regular case}
 In this subsection, we study the case when at least one  associated sequence   $\mathfrak s_{j_2,\ldots, j_n}=\{ \mathfrak{s}_{i,j_2,\ldots, j_n}\}_{i=0}^{\ell}$ is not regular. To do this, we use a simple approach based on the results of \cite{K20}.  
 
 \begin{definition}\label{20p.def4.4} (\cite{K20})  Let $\textbf{s}=\{s_j\}_{j=0}^{\ell}$  be a sequence of real numbers, let $\mathcal{N}(\textbf{s})=\{n_{j}\}_{j=1}^{N}$ be a set of normal indices of $\textbf{s}$ and $\alpha\in\mathbb{R}$.  $\textbf{s}=\{s_j\}_{j=0}^{\ell}$ is called  $\alpha-$regular sequence if
 \begin{equation}\label{20p.4.17}
 	P_{n_j}(\alpha)\neq 0\quad \mbox{for all } j=\overline{1, N}.
     \end{equation}

 \end{definition}
  \begin{theorem}\label{20p.th.4.5} Let $\mathbf{s}=\{s_{i_1,i_2,...,i_n}\}_{i_1,...,i_n=0}^{2\mathfrak{n}-1}$ be a sequence of real numbers and let the associated sequences   $\mathfrak s_{j_2,\ldots, j_n}=\{ \mathfrak{s}_{i,j_2,\ldots, j_n}\}_{i=0}^{2\mathfrak{n}-1}$ be  defined by~\eqref{21.1.13sw3}. Let $\mathcal{N}^{(j_2,...,j_n)}\left(\textbf{s}^{^{(j_2,...,j_n)}}\right)=\left\{n^{(j_2,...,j_n)}_{j}\right\}_{j=1}^{N_{j_2,...,j_n}}$ be the set of normal indices  of $\textbf{s}^{^{(j_2,...,j_n)}}$, where $n_{N_{j_2,...,j_n}}^{(j_2,...,j_n)}=\mathfrak{n}$ and let $A=
 \left\{\alpha_{j_2,...,j_n}\right\}_{j_2,...,j_n=0}^{2\mathfrak{n}-1}$ be a sequence of real numbers such that the associated sequence $ \mathfrak s_{j_2,\ldots, j_n}$ is $\alpha_{j_2,...,j_n}-$regular. Then
 any solution of the  problem $\bold{MP}(\mathbf{s}, 2\mathfrak{n}-1)$ takes the representation

  \begin{equation}\small\label{20p.4.18}\begin{split}
   &F(z_1,z_2,...,z_n)=  \sum_{j_2,\ldots, j_n=0}^{2\mathfrak{n}-1} \cfrac{1}{ \prod\limits_{i =2}^n\!z_i^{j_i+1}} \!\times\\&\!\times\!
      \frac{1}{\displaystyle\! -\!(z_1\!-\!\alpha_{j_2,...,j_n}) m^{A(j_2,\ldots, j_n)}_{1}\!(z_1)\!+\!\frac{1}{\displaystyle
    l^{A(j_2,\ldots, j_n)}_{1}\!+\!\cdots\!+\!\frac{1}{\!l^{A(j_2,\ldots, j_n)}_{N_{j_2,...,j_n}}\!(z_1)\!+\!\tau_{j_2,\ldots, j_n}(z_1)
    }}},
   \end{split}
 \end{equation}
  where the parameter $\tau_{j_2,\ldots, j_n}$ satisfies the following condition
 \begin{equation}\label{20p.4.19}
    \tau_{j_2,\ldots, j_n}(z_1)=o(1)
    \end{equation}
and atoms $\left(m^{A(j_2,\ldots, j_n)}_{j}, l^{A(j_2,\ldots, j_n)}_{j}\right)$ can be found   by
 \begin{equation}\label{20p.4.20}\begin{split}&
	d_1^{A(j_2,\ldots, j_n)}=\cfrac{1}{b_0^{(j_2,\ldots, j_n)}},\quad l_1^{A(j_2,\ldots, j_n)}=-\cfrac{1}{d_1^{(j_2,\ldots, j_n)}a_0^{(j_2,\ldots, j_n)}(\alpha_{j_2,...,j_n})},\\&
	d_j^{A(j_2,\ldots, j_n)}=\cfrac{1}{b_{j-1}^{(j_2,\ldots, j_n)}\left(l_{j-1}^{A(j_2,\ldots, j_n)}\right)^2 d_{j-1}^{A(j_2,\ldots, j_n)}},\\&
	l_j^{A(j_2,\ldots, j_n)}=-\cfrac{l_{j-1}^{A(j_2,\ldots, j_n)}}{1+l_{j-1}^{A(j_2,\ldots, j_n)}d_j^{A(j_2,\ldots, j_n)}a_{j-1}^{(j_2,\ldots, j_n)}(\alpha_{j_2,...,j_n})},\\&
	m_j^{A(j_2,\ldots, j_n)}(z_1)=\cfrac{\left(a_{j-1}^{(j_2,\ldots, j_n)}(z_1)-a_{j-1}^{(j_2,\ldots, j_n)}(\alpha_{j_2,...,j_n})\right)d_j^{A(j_2,\ldots, j_n)}}{z_1-\alpha_{j_2,...,j_n}},
\end{split}\end{equation}
where $d_j^{A(j_2,\ldots, j_n)}$ is the leading coefficient of the polynomial  $m_j^{A(j_2,\ldots, j_n)}$ and $\left(a^{(j_2,...,j_n)}_j,b^{(j_2,...,j_n)}_j\right)$ are defined by \eqref{20p.3.3}.

Furthermore, \eqref{20p.4.18} can be represented by
\begin{equation}\label{20p.4.21}
	F(z_1,z_2,...,z_n)\!=\!\! \!\! \sum_{j_2,\ldots, j_n=0}^{2\mathfrak{n}-1} \!\cfrac{1}{ \prod\limits_{i =2}^n\!\!z_i^{j_i+1}} \cfrac{Q^{+A(j_2,\ldots, j_n)}_{2N_{j_2,\ldots, j_n}-1}\!(z_1)\tau_{j_2,\ldots, j_n}(z_1)\!+\!Q^{+A(j_2,\ldots, j_n)}_{2N_{j_2,\ldots, j_n}}\!(z_1)}{P^{+A(j_2,\ldots, j_n)}_{2N_{j_2,\ldots, j_n}-1}\!(z_1)\tau_{j_2,\ldots, j_n}(z_1)\!+\!P^{+A(j_2,\ldots, j_n)}_{2N_{j_2,\ldots, j_n}}\!(z_1)},
\end{equation}
where $P^{+A(j_2,\ldots, j_n)}_{j}$ and $Q^{+A(j_2,\ldots, j_n)}_{j}$ are the Stieltjes polynomials with the shift $\alpha_{j_2,...,j_n}$ of the first and second kind associated with the sequence  $\mathfrak s_{j_2,\ldots, j_n}=\{ \mathfrak{s}_{i,j_2,\ldots, j_n}\}_{i=0}^{2\mathfrak{n}-1}$, respectively, $P^{+A(j_2,\ldots, j_n)}_{j}$ and $Q^{+A(j_2,\ldots, j_n)}_{j}$ are the solutions of the following system 
\begin{equation}\label{20p.4.22}
  \left\{
    \begin{array}{rcl}
    y_{2j}^{A(j_2,\ldots, j_n)}(z_1)-y^{A(j_2,\ldots, j_n)}_{2j-2}(z_1)=l^{A(j_2,\ldots, j_n)}_{j}(z_1)y^{A(j_2,\ldots, j_n)}_{2j-1}(z_1)\,,\\
    y^{A(j_2,\ldots, j_n)}_{2j+1}\!(z_1)\!-\!y^{A(j_2,\ldots, j_n)}_{2j-1}\!(z_1)\!=\!-\!(z_1-\alpha_{j_2,...,j_n})m^{A(j_2,\ldots, j_n)}_{j+1}\!(z_1)y^{(j_2,\ldots, j_n)}_{2j}\!(z_1)\\
    \end{array}\right.
\end{equation}
subject to the initial conditions
\begin{equation}\label{20p.4.23}\begin{split}&
   Q^{+A(j_2,\ldots, j_n)}_{-1}(z_1)\equiv1\quad\mbox{and}\quad Q^{+A(j_2,\ldots, j_n)}_{0}(z_1)\equiv0, \\&P^{+A(j_2,\ldots, j_n)}_{-1}(z_1)\equiv0\quad\mbox{and}\quad P^{+A(j_2,\ldots, j_n)}_{0}(z_1)\equiv1.
\end{split}\end{equation}
 \end{theorem}
\begin{proof} Assume the assumptions of statement hold. Hence, any solution of  $\bold{MP}(\mathbf{s}, 2\mathfrak{n}-1)$  admits  asymptotic expansion
\begin{equation}\label{20p.4.24}
   F(z_1,z_2,...,z_n)=  \sum_{j_2,\ldots, j_n=0}^{2\mathfrak{n}-1} \cfrac{1}{ \prod\limits_{i =2}^n\!\!z_i^{j_i+1}}F_{j_2,\ldots, j_n}(z_1),
\end{equation}
where $F_{j_2,\ldots, j_n}$ is associated with    $\mathfrak s_{j_2,\ldots, j_n}=\{ \mathfrak{s}_{i,j_2,\ldots, j_n}\}_{i=0}^{2\mathfrak{n}-1}$  and 
  \[
  	F_{j_2,\ldots, j_n}(z_1)=-\cfrac{ \mathfrak{s}_{0,j_2,\ldots, j_n}}{z_1}-\ldots-\cfrac{ \mathfrak{s}_{2n_{N_{j_2,\ldots, j_n}}-1,j_2,\ldots, j_n}}{z_1^{2n_{N_{j_2,\ldots, j_n}}}}+o\left(\cfrac{1}{z_1^{2n_{N_{j_2,\ldots, j_n}}}}\right).
  \]
  
 Due to  $\mathfrak s_{j_2,\ldots, j_n}=\{ \mathfrak{s}_{i,j_2,\ldots, j_n}\}_{i=0}^{2n_{N_{j_2,\ldots, j_n}}-1}$ is  $\alpha_{j_2,...,j_n}$ -- regular sequence and by Definition \ref{20p.def4.4}, we get
 \[
	 P^{(j_2,\ldots, j_n)}_{n^{(j_2,\ldots, j_n)}_j}(\alpha_{j_2,...,j_n})\neq 0\quad \mbox{for all } j=\overline{1, N_{j_2,...,j_n}}.
 \]

Moreover, by Theorem \ref{20p.th.3.1},  $F$ takes the form \eqref{20p.3.1}--\eqref{20p.3.2} and  we shift $F_{j_2,\ldots, j_n}$ by $\alpha_{j_2,...,j_n}$, we obtain
 \begin{equation}\label{20p.4.25}
 F_{j_2,\ldots, j_n}(z_1+\alpha_{j_2,...,j_n})=   \underset{i=0}{\overset{N_{j_2,...,j_n}-1}{\mathbf{K}}}\left(-\cfrac{b_{i}^{(j_2,...,j_n)}}{a_{i}^{(j_2,...,j_n)}(z_1+\alpha_{j_2,...,j_n})}\right).
 \end{equation}
Setting  $\tilde{F}_{j_2,\ldots, j_n}(z_1)= F_{j_2,\ldots, j_n}(z_1+\alpha_{j_2,...,j_n})$, $ \tilde{P}^{(j_2,\ldots, j_n)}_{n^{(j_2,\ldots, j_n)}_j}(z_1)=P^{(j_2,\ldots, j_n)}_{n^{(j_2,\ldots, j_n)}_j}(z_1+\alpha_{j_2,...,j_n})$ and  $\tilde{a}_{i}^{(j_2,...,j_n)}(z_1)=a_{i}^{(j_2,...,j_n)}(z_1+\alpha_{j_2,...,j_n})$, then  all polynomials  $ \tilde{P}^{(j_2,\ldots, j_n)}_{n^{(j_2,\ldots, j_n)}_j}$ do not vanish at $0$, i.e. $ \tilde{P}^{(j_2,\ldots, j_n)}_{n^{(j_2,\ldots, j_n)}_j}(0)\neq0$ and
 \begin{equation}\label{20p.4.26}
 \tilde{F}_{j_2,\ldots, j_n}(z_1)=   \underset{i=0}{\overset{N_{j_2,...,j_n}-1}{\mathbf{K}}}\left(-\cfrac{b_{i}^{(j_2,...,j_n)}}{\tilde{a}_{i}^{(j_2,...,j_n)}(z_1)}\right)
 \end{equation}
 is associated with a regular  sequence  $\tilde{\mathfrak s}_{j_2,\ldots, j_n}=\{ \tilde{\mathfrak s}_{i,j_2,\ldots, j_n}\}_{i=0}^{2n_{N_{j_2,\ldots, j_n}}-1}$. 
 
 By Theorem~\ref{20p.th.4.3}, we obtain
  \begin{equation}\footnotesize\label{20p.4.27}\begin{split}
   & \tilde{F}_{j_2,\ldots, j_n}(z_1)= 
      \frac{1}{\displaystyle\! -z_1 \tilde{m}^{(j_2,\ldots, j_n)}_{1}(z_1)\!+\!\frac{1}{\displaystyle
    l^{A(j_2,\ldots, j_n)}_{1}\!+\!\cdots\!+\!\frac{1}{\!{l}^{A(j_2,\ldots, j_n)}_{N_{j_2,...,j_n}}(z_1)\!+\!\tau_{j_2,\ldots, j_n}(z_1)}
    }}=\\&=
     \cfrac{\tilde{Q}^{+(j_2,\ldots, j_n)}_{2N_{j_2,\ldots, j_n}-1}(z_1)\tau_{j_2,\ldots, j_n}(z_1)\!+\!\tilde{Q}^{+(j_2,\ldots, j_n)}_{2N_{j_2,\ldots, j_n}}(z_1)}{\tilde{P}^{+(j_2,\ldots, j_n)}_{2N_{j_2,\ldots, j_n}-1}(z_1)\tau_{j_2,\ldots, j_n}(z_1)\!+\!\tilde{P}^{+(j_2,\ldots, j_n)}_{2N_{j_2,\ldots, j_n}}(z_1)},
   \end{split}
 \end{equation}
  where the parameter $\tau_{j_2,\ldots, j_n}$ satisfies the following
 \begin{equation}\label{20p.4.28}
    \tau_{j_2,\ldots, j_n}(z_1)=o(1),
  \end{equation}
the atoms $\left(\tilde{m}^{(j_2,\ldots, j_n)}_{j},l^{A(j_2,\ldots, j_n)}_{j}\right)$ can be calculated    by 
\begin{equation}\label{20p.4.29}\begin{split}&
	d_1^{(j_2,\ldots, j_n)}=\cfrac{1}{b_0^{(j_2,\ldots, j_n)}},\quad l_1^{A(j_2,\ldots, j_n)}=-\cfrac{1}{d_1^{(j_2,\ldots, j_n)}\tilde{a}_0^{(j_2,\ldots, j_n)}(0)},\\&
	d_j^{(j_2,\ldots, j_n)}=\cfrac{1}{b_{j-1}^{(j_2,\ldots, j_n)}\left(l_{j-1}^{A(j_2,\ldots, j_n)}\right)^2 d_{j-1}^{(j_2,\ldots, j_n)}},\\&
	l_j^{A(j_2,\ldots, j_n)}=-\cfrac{l_{j-1}^{A(j_2,\ldots, j_n)}}{1+l_{j-1}^{A(j_2,\ldots, j_n)}d_j^{(j_2,\ldots, j_n)}\tilde{a}_{j-1}^{(j_2,\ldots, j_n)}(0)},\\&
	\tilde{m}_j^{(j_2,\ldots, j_n)}(z_1)=\cfrac{\left(\tilde{a}_{j-1}^{(j_2,\ldots, j_n)}(z_1)-\tilde{a}_{j-1}^{(j_2,\ldots, j_n)}(0)\right)d_j^{(j_2,\ldots, j_n)}}{z_1},
\end{split}\end{equation}
 $\tilde{P}^{+(j_2,\ldots, j_n)}_{j}$ and $\tilde{Q}^{+(j_2,\ldots, j_n)}_{j}$ are the solutions of the  similar system~\eqref{20p.4.5}--\eqref{20p.4.6}.
 
 Using backward substitution, i.e. 
 \[\begin{split}
 &F_{j_2,\ldots, j_n}(z_1)=\tilde{F}_{j_2,\ldots, j_n}(z_1-\alpha_{j_2,...,j_n}),\\&{m}_j^{A(j_2,\ldots, j_n)}(z_1)=\tilde{m}_j^{(j_2,\ldots, j_n)}(z_1-\alpha_{j_2,...,j_n}),\\&
{P}^{+A(j_2,\ldots, j_n)}_{j}(z_1)= \tilde{P}^{+(j_2,\ldots, j_n)}_{j}(z_1-\alpha_{j_2,...,j_n}),\\&
{Q}^{+A(j_2,\ldots, j_n)}_{j}(z_1)= \tilde{Q}^{+(j_2,\ldots, j_n)}_{j}(z_1-\alpha_{j_2,...,j_n}),
\end{split} \]
 we obtain \eqref{20p.4.18}--\eqref{20p.4.23}. This completes the proof.~\end{proof}
 \begin{theorem}\label{20p.th.4.6} Let $\mathbf{s}=\{s_{i_1,i_2,...,i_n}\}_{i_1,...,i_n=0}^{2\mathfrak{n}-2}$ be a sequence of real numbers and let the associated sequences   $\mathfrak s_{j_2,\ldots, j_n}\!=\!\{\! \mathfrak{s}_{i,j_2,\ldots, j_n}\}_{i=0}^{2\mathfrak{n}-2}$ be  defined by~\eqref{21.1.13sw3}. Let $\mathcal{N}^{(j_2,...,j_n)}\!\left(\!\textbf{s}^{^{(j_2,...,j_n)}}\!\right)\!\!=\!\left\{n^{(j_2,...,j_n)}_{j}\right\}_{j=1}^{N_{j_2,...,j_n}}$ be the set of normal indices  of $\textbf{s}^{^{(j_2,...,j_n)}}$, where $n_{N_{j_2,...,j_n}}^{(j_2,...,j_n)}=\mathfrak{n}$ and let $A=
 \left\{\alpha_{j_2,...,j_n}\right\}_{j_2,...,j_n=0}^{2\mathfrak{n}-2}$ be a sequence of real numbers such that the associated sequence $ \mathfrak s_{j_2,\ldots, j_n}$ is $\alpha_{j_2,...,j_n}-$regular. Then
 any solution of the  problem $\bold{MP}(\mathbf{s}, 2\mathfrak{n}-2)$ admits the representation

  \begin{equation}\footnotesize\label{20p.4.30}\begin{split}
   &F(z_1,z_2,...,z_n)=  \sum_{j_2,\ldots, j_n=0}^{2\mathfrak{n}-2} \cfrac{1}{ \prod\limits_{i =2}^n\!z_i^{j_i+1}} \!\times\\&\!\times\!
      \frac{1}{\displaystyle\! -\!(z_1\!-\!\alpha_{j_2,...,j_n}) m^{A(j_2,\ldots, j_n)}_{1}\!(z_1)\!+\!\frac{1}{\displaystyle
    l^{A(j_2,\ldots, j_n)}_{1}\!+\!\cdots\!+\!\frac{1}{\! -\!(z_1\!-\!\alpha_{j_2,...,j_n}) m^{A(j_2,\ldots, j_n)}_{N_{j_2,...,j_n}}\!(z_1)\!+\!\tau^{-1}_{j_2,\ldots, j_n}(z_1)
    }}},
   \end{split}
 \end{equation}
  where the parameter $\tau_{j_2,\ldots, j_n}$ satisfies the following condition
 \begin{equation}\label{20p.4.31}
   \frac{1}{ \tau_{j_2,\ldots, j_n}(z_1)}=o(z_1)
    \end{equation}
and atoms $\left(m^{A(j_2,\ldots, j_n)}_{j}, l^{A(j_2,\ldots, j_n)}_{j}\right)$ can be calculated by \eqref{20p.4.20}.

Furthermore, \eqref{20p.4.30} can be represented by
\begin{equation}\label{20p.4.32}
	F(z_1,z_2,...,z_n)\!=\!\! \!\! \sum_{j_2,\ldots, j_n=0}^{2\mathfrak{n}-2} \!\cfrac{1}{ \prod\limits_{i =2}^n\!\!z_i^{j_i+1}} \cfrac{Q^{+A(j_2,\ldots, j_n)}_{2N_{j_2,\ldots, j_n}-1}\!(z_1)\tau_{j_2,\ldots, j_n}(z_1)\!+\!Q^{+A(j_2,\ldots, j_n)}_{2N_{j_2,\ldots, j_n}-2}\!(z_1)}{P^{+A(j_2,\ldots, j_n)}_{2N_{j_2,\ldots, j_n}-1}\!(z_1)\tau_{j_2,\ldots, j_n}(z_1)\!+\!P^{+A(j_2,\ldots, j_n)}_{2N_{j_2,\ldots, j_n}-2}\!(z_1)},
\end{equation}
where  $P^{+A(j_2,\ldots, j_n)}_{j}$ and $Q^{+A(j_2,\ldots, j_n)}_{j}$ are the Stieltjes polynomials with the shift $\alpha_{j_2,...,j_n}$ defined \eqref{20p.4.22}--\eqref{20p.4.23}.

 \end{theorem}
 
 \begin{proof} Let the assumptions of  theorem hold. The proof  is based on the same approach as in the previous theorem. Any solution of  $\bold{MP}(\mathbf{s}, 2\mathfrak{n}-2)$  takes the form
\begin{equation}\label{20p.4.33}
   F(z_1,z_2,...,z_n)=  \sum_{j_2,\ldots, j_n=0}^{2\mathfrak{n}-2} \cfrac{1}{z_2^{j_2+1}\cdot\ldots\cdot z_n^{j_n+1}}F_{j_2,\ldots, j_n}(z_1),
\end{equation}
where $F_{j_2,\ldots, j_n}$ is associated with     $\alpha_{j_2,...,j_n}$ -- regular sequence $\mathfrak s_{j_2,\ldots, j_n}=\{ \mathfrak{s}_{i,j_2,\ldots, j_n}\}_{i=0}^{2\mathfrak{n}-2}$. By \cite[Theorem 5.1]{K20}, we obtain
\begin{equation}\label{20p.4.34}\footnotesize\begin{split}&
F_{j_2,\ldots, j_n}(z_1)=  \\&= \frac{1}{\displaystyle\! -\!(z_1\!-\!\alpha_{j_2,...,j_n}) m^{A(j_2,\ldots, j_n)}_{1}\!(z_1)\!+\!\cdots\!+\!\frac{1}{\! -\!(z_1\!-\!\alpha_{j_2,...,j_n}) m^{A(j_2,\ldots, j_n)}_{N_{j_2,...,j_n}}\!(z_1)\!+\!\tau^{-1}_{j_2,\ldots, j_n}(z_1)
    }},\end{split}
\end{equation}
where  $\left(m^{A(j_2,\ldots, j_n)}_{j}, l^{A(j_2,\ldots, j_n)}_{j}\right)$ can be found by \eqref{20p.4.20} and $\tau_{j_2,\ldots, j_n}$ satisfies~\eqref{20p.4.31}. On the other hand, \cite[Theorem 5.5]{K20}, we get
\begin{equation}\label{20p.4.35}
F_{j_2,\ldots, j_n}(z_1)=\cfrac{Q^{+A(j_2,\ldots, j_n)}_{2N_{j_2,\ldots, j_n}-1}\!(z_1)\tau_{j_2,\ldots, j_n}(z_1)\!+\!Q^{+A(j_2,\ldots, j_n)}_{2N_{j_2,\ldots, j_n}-2}\!(z_1)}{P^{+A(j_2,\ldots, j_n)}_{2N_{j_2,\ldots, j_n}-1}\!(z_1)\tau_{j_2,\ldots, j_n}(z_1)\!+\!P^{+A(j_2,\ldots, j_n)}_{2N_{j_2,\ldots, j_n}-2}\!(z_1)}, 
\end{equation}
where $P^{+A(j_2,\ldots, j_n)}_{j}$ and $Q^{+A(j_2,\ldots, j_n)}_{j}$ are the Stieltjes polynomials with the shift $\alpha_{j_2,...,j_n}$, which  can be calculated by \eqref{20p.4.22}--\eqref{20p.4.23}.
Substituting \eqref{20p.4.34}--\eqref{20p.4.35} into \eqref{20p.4.33}, we obtain \eqref{20p.4.30}--\eqref{20p.4.32}. This completes the proof.~\end{proof}

\begin{remark} Similarly\cite{K20},  $P^{+A(j_2,\ldots, j_n)}_{j}$ and $Q^{+A(j_2,\ldots, j_n)}_{j}$ can be represented in terms of the polynomials   $P^{(j_2,\ldots, j_n)}_{j}$ and $Q^{(j_2,\ldots, j_n)}_{j}$ by
\begin{equation}\label{20p.4.36}\small
    \begin{split}
        &P^{+A(j_2,\ldots, j_n)}_{-1}(z_1)\equiv0\quad\mbox{and}\quad P^{+A(j_2,\ldots, j_n)}_{0}(z_1)\equiv1,\,\\&
        Q^{+A(j_2,\ldots, j_n)}_{-1}(z_1)\equiv1\quad\mbox{and}\quad  Q^{+A(j_2,\ldots, j_n)}_{0}(z_1)\equiv0,
        \\
        &P^{+A(j_2,\ldots, j_n)}_{2i-1}(z_1)=\frac{-1}{b^{(j_2,\ldots, j_n)}_{0}\ldots b^{(j_2,\ldots, j_n)}_{i-1}}
        \begin{vmatrix}
            P^{(j_2,\ldots, j_n)}_{n_{i}}(z_1) & P^{(j_2,\ldots, j_n)}_{n_{i-1}}(z_1) \\
            P^{(j_2,\ldots, j_n)}_{n_{i}}(\alpha_{j_2,...,j_n}) & P^{(j_2,\ldots, j_n)}_{n_{i-1}}(\alpha_{j_2,...,j_n})\\
        \end{vmatrix},\\&
        Q_{2i-1}^{+A(j_2,\ldots, j_n)}(z_1)=\frac{1}{b^{(j_2,\ldots, j_n)}_{0}\ldots b^{(j_2,\ldots, j_n)}_{i-1}}
        \begin{vmatrix}
            Q^{(j_2,\ldots, j_n)}_{n_{i}}(z_1) & Q^{(j_2,\ldots, j_n)}_{n_{i-1}}(z_1) \\
            P^{(j_2,\ldots, j_n)}_{n_{i}}(\alpha_{j_2,...,j_n}) & P^{(j_2,\ldots, j_n)}_{n_{i-1}}(\alpha_{j_2,...,j_n})\\
        \end{vmatrix},\\&  P^{+A(j_2,\ldots, j_n)}_{2i}(z_1)=\frac{P^{(j_2,\ldots, j_n)}_{n_{i}}(z_1)}{P^{(j_2,\ldots, j_n)}_{n_{i}}(\alpha_{j_2,...,j_n})},\,
        Q^{+A(j_2,\ldots, j_n)}_{2i}(z_1)=-\frac{Q^{(j_2,\ldots, j_n)}_{n_{i}}(z_1)}{P^{(j_2,\ldots, j_n)}_{n_{i}}(\alpha_{j_2,...,j_n})}.
    \end{split}
\end{equation}

\end{remark}

 \section{Full problem}
 
 In this section, we study  a full problem, which can be formulated by
 
 Given a sequence of real numbers  $\mathbf{s}=\{s_{i_1,i_2,...,i_n}\}_{i_1,...,i_n=0}^{\infty}$. Find a description of a set of functions F such that
 \begin{equation}\label{20p. 5.1}
 	 F(z_1,z_2,...,z_n)=  -\sum_{k=0}^{\infty} \sum_{\begin{matrix} 
  \alpha_i^k\geq0  \\
  \alpha_1^k+\ldots \alpha_n^k=k \\
   \end{matrix} }\binom{k}{\alpha_1^k, \ldots, \alpha_n^k}\cfrac{s_{\alpha_1^k,\alpha_2^k,\ldots, \alpha_n^k}}{z_1^{\alpha_1^k+1}z_2^{\alpha_2^k+1}\ldots z_n^{\alpha_n^k+1}}.
  \end{equation}
 
On the other hand, \eqref{20p. 5.1} can be represented in terms of the associated sequences $\mathfrak s_{j_2,\ldots, j_n}=\{ \mathfrak{s}_{i,j_2,\ldots, j_n}\}_{i=0}^{\infty}$ as follows
  \begin{equation}\label{20p. 5.2}\begin{split}
   	 F(z_1,z_2,...,z_n)&=  -\sum_{j_2,\ldots, j_n=0}^\infty\cfrac{1}{z_2^{j_2+1}\cdot\ldots\cdot z_n^{j_n+1}}F_{j_2,\ldots, j_n}(z_1),  
   \end{split}
 \end{equation}
  where   $F_{j_2,\ldots, j_n}$  admits the  asymptotic expansion
    \begin{equation}\label{20p. 5.3}
    	F_{j_2,\ldots, j_n}(z_1)=-\cfrac{ \mathfrak{s}_{0,j_2,\ldots, j_n}}{z_1}-\cfrac{ \mathfrak{s}_{1,j_2,\ldots, j_n}}{z_1^2}-\ldots-\cfrac{ \mathfrak{s}_{\ell,j_2,\ldots, j_n}}{z_1^{\ell+1}}+\ldots=-\sum_{j=0}^{\infty}\cfrac{ \mathfrak{s}_{j,j_2,\ldots, j_n}}{z^{j+1}_1}
     \end{equation}
and $ \mathfrak{s}_{j,j_2,\ldots, j_n}$ are defined by \eqref{21.1.13sw3}. The full problem is denoted by  $\bold{MP}(\mathbf{s})$.

  \begin{theorem}\label{20p.th.5.1} Let $\mathbf{s}=\{s_{i_1,i_2,...,i_n}\}_{i_1,...,i_n=0}^{\infty}$ be a sequence of real numbers and let  all associated sequences $\mathfrak s_{j_2,\ldots, j_n}=\{ \mathfrak{s}_{i,j_2,\ldots, j_n}\}_{i=0}^{\infty}$ be regular, where $ \mathfrak{s}_{i,j_2,\ldots, j_n}$ be defined by~\eqref{21.1.13sw3}. Then
 $\bold{MP}(\mathbf{s})$ is indeterminate if and only if 
 \begin{equation}\label{20p. 5.4}
 	\sum_{j_2,\ldots, j_n=0}^{\infty} \sum_{j=1}^{\infty}m_j^{(j_2,\ldots, j_n)}(0)<\infty\mbox{ and }
	\sum_{j_2,\ldots, j_n=0}^{\infty} \sum_{j=1}^{\infty}l_j^{(j_2,\ldots, j_n)}<\infty.
\end{equation}

Furthermore, if \eqref{20p. 5.4} holds, then the following limits converge and 
\begin{equation}\label{20p. 5.5}\begin{split}
 \lim _{N_{j_2,\ldots, j_n} \rightarrow\infty} &\cfrac{Q^{+(j_2,\ldots, j_n)}_{2N_{j_2,\ldots, j_n}-1}(z_1)\tau_{j_2,\ldots, j_n}(z_1)\!+\!Q^{+(j_2,\ldots, j_n)}_{2N_{j_2,\ldots, j_n}}(z_1)}{P^{+(j_2,\ldots, j_n)}_{2N_{j_2,\ldots, j_n}-1}(z_1)\tau_{j_2,\ldots, j_n}(z_1)\!+\!P^{+(j_2,\ldots, j_n)}_{2N_{j_2,\ldots, j_n}}(z_1)}=\\&=\cfrac{w^{+(j_2,\ldots, j_n)}_{11}(z_1)\tau_{j_2,\ldots, j_n}(z_1)\!+\!w^{+(j_2,\ldots, j_n)}_{12}(z_1)}{w^{+(j_2,\ldots, j_n)}_{21}(z_1)\tau_{j_2,\ldots, j_n}(z_1)\!+\!w^{+(j_2,\ldots, j_n)}_{22}(z_1)}.\end{split}
\end{equation}
Any solution of  $\bold{MP}(\mathbf{s})$ can be represented as 
\begin{equation}\label{20p. 5.6}
    	 F(z_1,z_2,...,z_n)=
	 \sum_{j_2,\ldots, j_n=0}^{\infty} \cfrac{w^{+(j_2,\ldots, j_n)}_{11}(z_1)\tau_{j_2,\ldots, j_n}(z_1)\!+\!w^{+(j_2,\ldots, j_n)}_{12}(z_1)}{ \prod\limits_{i =2}^n\!\!z_i^{j_i+1}\!\!\left(w^{+(j_2,\ldots, j_n)}_{21}\!(z_1)\tau_{j_2,\ldots, j_n}(z_1)\!+\!w^{+(j_2,\ldots, j_n)}_{22}\!(z_1)\right)},
\end{equation}
where the parameters $\tau_{j_2,\ldots, j_n}(z_1)=o(1)$.
 \end{theorem}
  
 \begin{proof} Let the assumptions of  theorem hold. The proof is based on the results of \cite{DK20}. Any solution of the full problem $\bold{MP}(\mathbf{s})$ takes the form \eqref{20p. 5.2}--\eqref{20p. 5.3}. Hence, $F_{j_2,\ldots, j_n}$ is a solution of the one-dimensional full probelem  $\bold{MP}(\mathfrak s_{j_2,\ldots, j_n})$, where $\mathfrak s_{j_2,\ldots, j_n}=\{ \mathfrak{s}_{i,j_2,\ldots, j_n}\}_{i=0}^{\infty}$ is regular sequence and  defined by~\eqref{21.1.13sw3}. In this case, the proof is based on the results of \cite{DK15} and \cite{DK20}. First of all, we get representation for $F_{j_2,\ldots, j_n}$. By \cite{DK15}, we obtain 
 \begin{equation}\label{20p. 5.7}
 F_{j_2,\ldots, j_n}(z_1)=   \frac{1}{\displaystyle\! -z_1 m^{(j_2,\ldots, j_n)}_{1}(z_1)\!+\!\frac{1}{\displaystyle
    l^{(j_2,\ldots, j_n)}_{1}\!+\ldots}},
 \end{equation}
 where $\left(m^{(j_2,\ldots, j_n)}_{j}, l^{(j_2,\ldots, j_n)}_{j}\right)$ can be calculated by~\eqref{20p.4.3}. Moreover, by \cite{DK20},  $\bold{MP}(\mathfrak s_{j_2,\ldots, j_n})$ is indeterminate if and only if 
  \begin{equation}\label{20p. 5.8}
\sum_{j=1}^{\infty}m_j^{(j_2,\ldots, j_n)}(0)<\infty\mbox{ and }
	\sum_{j=1}^{\infty}l_j^{(j_2,\ldots, j_n)}<\infty.
  \end{equation}
 Then the limit \eqref{20p. 5.8} exists and 
 \begin{equation}\label{20p. 5.9}
 F_{j_2,\ldots, j_n}(z_1)=\cfrac{w^{+(j_2,\ldots, j_n)}_{11}(z_1)\tau_{j_2,\ldots, j_n}(z_1)\!+\!w^{+(j_2,\ldots, j_n)}_{12}(z_1)}{w^{+(j_2,\ldots, j_n)}_{21}\!(z_1)\tau_{j_2,\ldots, j_n}(z_1)\!+\!w^{+(j_2,\ldots, j_n)}_{22}\!(z_1)}.
  \end{equation}
 
 Therefore, $\bold{MP}(\mathbf{s})$  is indeterminate if and only if all problems $\bold{MP}(\mathfrak s_{j_2,\ldots, j_n})$ are indeterminate, i.e.  \eqref{20p. 5.4} holds. Substituting \eqref{20p. 5.8} into \eqref{20p. 5.2}, we obtain \eqref{20p. 5.6}. This completes the proof.~\end{proof}

\textbf{Acknowledgements} The author is grateful to Prof. K. Schmüdgen for introducing him to the multidimensional problem of moments.

 
 
 
 
 
 
 
 
 
 

 
 
 
 

\end{document}